\documentclass{amsart}
\usepackage{kotex}
\usepackage{CJKutf8}
\usepackage{amsmath}
\usepackage{amssymb}
\usepackage{enumerate}
\usepackage{amsthm}
\usepackage[all]{xy}
\usepackage{tikz}
\usetikzlibrary{arrows.meta,arrows}
\usepackage{hyperref}
\usepackage{esint}

\newcommand{\C}{\mathbb{C}}
\newcommand{\Z}{\mathbb{Z}}
\newcommand{\R}{\mathbb{R}}

\newcommand{\Pp}{\mathbb{P}}

\newcommand{\Hh}{\mathbb{H}}

\newcommand{\ACTS}{\curvearrowright}

\newtheorem{lemma}{Lemma}
\newtheorem{corollary}[lemma]{Corollary}
\newtheorem{theorem}[lemma]{Theorem}
\newtheorem{proposition}[lemma]{Proposition}
\newtheorem*{extTheorem}{Theorem} 
\newtheorem*{definition}{Definition}
\theoremstyle{definition}

\title{Kummer Rigidity for Hyperk\"ahler Automorphisms}
\author{Seung uk Jang}
\date{July 26, 2021. Last revised \today}

\begin{document}

\begin{abstract}
  We show that a holomorphic automorphism on a projective hyperk\"ahler manifold that has positive topological entropy and has volume measure as the measure of maximal entropy, is necessarily a Kummer example. This partially extends the analogous results in \cite{CantatDupont2020}\cite{FT18} for complex surfaces.
  
  A trick with Jensen's inequality is used to show that stable and unstable distributions exhibit uniform rate of contraction and expansion, and with them our hyperk\"ahler manifold is shown to be flat, modulo contracting some loci. A result in \cite{GKP} then implies that our hyperk\"ahler manifold is birational to a torus quotient, giving the Kummer example structure.
\end{abstract}

\maketitle

\section{Introduction}

In this paper, we are interested in some rigidity results on holomorphic automorphisms ${f}\colon X\to X$ of a projective hyperk\"ahler manifold $X$. By a hyperk\"ahler manifold $X$ we mean a compact K\"ahler manifold $X$ which is simply connected and the group $H^{2,0}(X,\C)$ is generated by an everywhere nondegenerate holomorphic 2-form $\Omega$ \cite{oguiso2009}. Hyperk\"ahler manifolds are of interest from the classification result by Beauville \cite[Th\'eor\`eme 1]{Beauville83}\cite[Proposition 1]{Beauville85} for compact K\"ahler manifolds with the zero first Chern class.

\subsection{Hyperk\"ahler Kummer Examples}

One of the first examples of hyperk\"ahler manifolds by Beauville \cite[\S{}7]{Beauville83} is constructed from complex tori, now known as generalized Kummer varieties \cite[\S{}21.2]{ghj01}. Likewise, one can consider a holomorphic automorphism $(X,{f})$ on a hyperk\"ahler manifold which is constructed from that on a complex torus. Extending this idea, we use the term \emph{Kummer example}, following \cite[Definition 1.3]{CantatDupont2020}:
\begin{definition}[Kummer Example]
  Let $X$ be a hyperk\"ahler manifold and let ${f}$ be a holomorphic automorphism of $X$. The pair $(X,{f})$ is a \emph{(hyperk\"ahler) Kummer example} if there exist
\begin{itemize}
\item a birational morphism $\phi\colon X\to Y$ onto an orbifold $Y$,
\item a finite orbifold cover $\pi\colon \mathbb{T}\to Y$ by a complex torus $\mathbb{T}$, and
\item an automorphism $\widetilde{f}$ on $Y$ and an automorphism $A$ on $\mathbb{T}$ such that
\[\widetilde{f}\circ\phi = \phi\circ {f}\quad\text{ and }\quad \widetilde{f}\circ\pi=\pi\circ A.\]
\end{itemize}
\end{definition}

Although the word `orbifold' and `orbifold cover' may vary among the literatures, we use the terms as in \cite[\S{13.2}]{Thurston}. By ibid., $Y$ is good in the sense that it is given by a quotient $\mathbb{T}/\Gamma$ by a finite group $\Gamma$. The description of $Y$ in Theorem \ref{lem:00} below is thus general enough to cover the orbifolds of interest.

\subsection{Statement of the Result}

In case if $X$ is a projective complex surface or a K3 surface, results like \cite[\S{1.3} Main Theorem]{CantatDupont2020} or \cite[Theorem 1.1.1]{FT18} imply that, if the (Riemannian) volume measure on $X$ is a measure of maximal entropy for the dynamics $(X,{f})$, then the pair $(X,{f})$ is necessarily a Kummer example. The aim of this paper is to generalize this to the projective hyperk\"ahler case.

For the statement of the main result below, we recall the following. A hyperk\"ahler manifold $X$ of dimension $\dim_\C X=2n$, with a generator $\Omega$ of the group $H^{2,0}(X,\C)$, carries a Ricci-flat K\"ahler metric $\omega$ whose Riemannian volume $\omega^{2n}$ equals to the $(2n,2n)$-form $(\Omega\wedge\overline{\Omega})^n$ \cite[Theorem 5.11]{ghj01}. We call the measure defined by this form the \emph{volume measure} of $X$.

\begin{theorem}
  \label{lem:00} 
  Let $X$ be a projective hyperk\"ahler manifold. Let ${f}\colon X\to X$ be a holomorphic automorphism that has positive topological entropy.
  Suppose the volume measure is the ${f}$-invariant measure of maximal entropy.
  Then the underlying hyperk\"ahler manifold $X$ is a normalization of a torus
  quotient, and ${f}$ is induced from a hyperbolic affine-linear transformation on that
  torus quotient.

  That is, if $\dim_\C X=2n$, there exists a complex torus $\mathbb{T}=\C^{2n}/\Lambda$ and a 
  finite group of toral isomorphisms $\Gamma$ 
  in which $X$ normalizes $\mathbb{T}/\Gamma$, and ${f}$ is induced from a 
  hyperbolic affine-linear transformation 
  $A\colon\mathbb{T}/\Gamma\to\mathbb{T}/\Gamma$.
\end{theorem}

\subsection{Examples of Hyperk\"ahler Kummer Examples}

Before going through the proof of the main theorem, we list some examples that a reader may have in mind. These examples are brought from \cite[\S{}3.3-3.4]{LoB17}, and we seek for whether each example
\begin{itemize}
  \item is actually a (hyperk\"ahler) Kummer example, and
  \item has the volume as a measure of maximal entropy.
\end{itemize}

Throughout this section, following \cite[\S{}3]{LoB17}, we denote $T$ as a 2-dimensional complex torus, $f_T\colon T\to T$ a hyperbolic automorphism on it; by hyperbolic we mean by $h_{\mathrm{top}}(f_T)>0$ (cf. \cite[Corollary 1.23]{LoB17}). Let $f_T^{\times n}=(f_T,\cdots,f_T)\colon T^n\to T^n$ be the product of $n$ copies of $f_T$. From \cite[Lemma 3.1]{LoB17}, it is known that $f_T^{\times n}$ has unstable and stable foliations $\mathcal{F}^+$ and $\mathcal{F}^-$, obtained by making the $n$-product of those for $(T,f_T)$.
The topological entropy of $(T^n,f_T^{\times n})$ is $n\cdot h_{\mathrm{top}}(f_T)$; cf. \cite[Proposition 3.1.7(4)]{KatokHasselblatt}.

Moreover, combining \cite[Proposition 1.12]{LoB17} and a theorem of Gromov--Yomdin \cite[Theorem 1.10]{LoB17}, a Kummer example $(X,f)$ has the same topological entropy as its associated toral automorphism $(\mathbb{T},A)$. That is, $h_{\mathrm{top}}(X,f)=h_{\mathrm{top}}(\mathbb{T},A)$.

\subsubsection{Generalized Kummer Variety}
\label{subsec:gen-kummer-var}

Denote $K_n(T)$ for the $2n$-dimensional generalized Kummer variety, the notation following \cite[\S{}21.2]{ghj01}. To see how it is built, consider the Hilbert scheme $T^{[n+1]}$ of $(n+1)$ points in $T$ and the $(n+1)$-th symmetric product $\mathrm{Sym}^{n+1}(T)$ of $T$. Consider the map $T^{[n+1]}\to\mathrm{Sym}^{n+1}(T)\to T$, $Z\mapsto\{p_i\}\mapsto\sum p_i$. The fiber $\subset T^{[n+1]}$ of this map at $0\in T$ is called the \emph{generalized Kummer variety} $K_n(T)$.

Following \cite[\S{}3.4]{LoB17}, $f_T$ induces an automorphism $K_n(f_T)\colon K_n(T)\to K_n(T)$. By \cite[Lemma 3.12]{LoB17}, the pair $(X,f)=(K_n(T),K_n(f_T))$ is a Kummer example, by the following data:
\begin{itemize}
  \item $(Y,\widetilde{f})=(T^n/\mathfrak{S}_{n+1},f_T^{\times n}/\mathfrak{S}_{n+1})$,
  \item $(\mathbb{T},A)=(T^n,f_T^{\times n})$, and
  \item the quotient map $q\colon\mathbb{T}\to Y$, which is birationally equivalent to a generically finite meromorphic map $\pi\colon\mathbb{T}\dashrightarrow X$. (The birational equivalence $\phi\colon X\to Y$ may be defined on whole $X$.)
\end{itemize}
Here, $\mathfrak{S}_{n+1}$ is the symmetry group of $(n+1)$ letters, acting on $T^n$ by restricting the natural action $\mathfrak{S}_{n+1}\ACTS T^{n+1}$ on
\[T^n=\{(t_0,t_1,\cdots,t_n)\in T^{n+1}\mid t_0+t_1+\cdots+t_n=0\}.\]
Therefore $(X,f)$ is a Kummer example, with its associated toral automorphism $(T^n,f_T^{\times n})$.

By loc.cit., induced from foliations $\mathcal{F}^+$ and $\mathcal{F}^-$ on $T^n$, we have foliations $\mathcal{F}^+_X$ and $\mathcal{F}^-_X$ on $X$, called unstable and stable foliations respectively. (Here $\mathcal{F}^\pm_X$ may have singular loci.) By the action of $f_T^{\times n}$, each vector tangent to the foliations $\mathcal{F}^+$ and $\mathcal{F}^-$ on $T^n$ are dilated by $e^{h/2}$ and $e^{-h/2}$ respectively, where $h=h_{\mathrm{top}}(f_T)$. The same rates apply for $\mathcal{F}^+_X$ and $\mathcal{F}^-_X$ on $X\setminus\mathrm{Sing}(\mathcal{F}^\pm_X)$.

This gives that the Lyapunov exponents of $(X,f)$ under the volume is $\pm h/2$, with multiplicity $n$ each. The Ledrappier--Young formula \cite[Corollary G]{LY85II} then yields $h_{\mathrm{vol}}(X,f)=nh$. Now since $(X,f)$ is a Kummer example, its topological entropy is also $nh$, as that of $(T^n,f_T^{\times n})$ is $nh$. Hence the volume measure is a measure of maximal entropy.

\subsubsection{Hilbert Scheme of a Kummer Surface}

Denote $K_1(T)$ for the Kummer surface of the 2-dimensional complex torus $T$. Then one has the Hilbert scheme $X:=K_1(T)^{[n]}$ of $n$ points, which is hyperk\"ahler, and the induced map $f:=K_1(f_T)^{[n]}\colon X\to X$, from $f_T\colon T\to T$ (cf. \cite[Proposition 3.10]{LoB17}).

The hyperk\"ahler manifold $X=K_1(T)^{[n]}$ is obtained by normalizing $T^n/\Gamma$, where $\Gamma$ is generated by
\begin{itemize}
  \item[-] an involution $\theta\colon T^n\to T^n$, $\theta(t_1,t_2,\cdots,t_n)=(-t_1,t_2,\cdots,t_n)$, and
  \item[-] the symmetry group $\mathfrak{S}_n\ACTS T^n$ on coordinates.
\end{itemize}
(The group $\Gamma$, generated by involutions, forms the Weyl group of the Lie algebra $B_n$.) The map $f_T^{\times n}$ commutes with $\Gamma$, thus induces a map $\widetilde{f}\colon T^n/\Gamma\to T^n/\Gamma$. The map $f=K_1(f_T)^{[n]}$ then satisfies, with the normalization map $\phi\colon X\to T^n/\Gamma$, $\widetilde{f}\circ\phi=\phi\circ f$. Thus $(X,f)$ is a Kummer example, with associated toral automorphism $(T^n,f_T^{\times n})$.

The foliations $\mathcal{F}^\pm$ on $T^n$ are $\Gamma$-invariant, hence carrying this to the regular locus of $T^n/\Gamma$ and inducing (singular) foliations on $X$, we obtain unstable and stable foliations $\mathcal{F}^+_X$ and $\mathcal{F}^-_X$. Arguing as in \S{\ref{subsec:gen-kummer-var}}, we see that for $(X,f)$, the volume measure is a measure of maximal entropy.

\subsection{Outline of the Proof}

The key points of our main result (Theorem \ref{lem:00}) can be summarized as follows. First, under the assumption of positive topological entropy, the holomophic automorphism $f$ can be shown to be (non-uniformly) hyperbolic. Hence, denoting $\mu$ for the measure of maximal entropy, the stable $E^-$ and unstable $E^+$ distributions can be defined $\mu$-almost everywhere.

Thanks to the hyperk\"ahler structure of the underlying manifold $X$, we can extract the following dynamical data.
\begin{enumerate}
  \item We have a metric $\omega_0$ defined on the complement of a proper $f$-invariant analytic subvariety $E\subset X$, which ``dilates'' along the dynamics. (See \S{\ref{sec:proof-step-1}} for details.)
\end{enumerate}

To elaborate, we have an appropriate rate $\lambda>1$ (see \eqref{eqn:h-and-lambda}) such that, for all $v$ in the unstable distribution, we have $(f^\ast\omega_0)(v,v)=\lambda\cdot\omega_0(v,v)$, and for all $v$ in the stable distribution, we have $(f^\ast\omega_0)(v,v)=\lambda^{-1}\cdot\omega_0(v,v)$. In other words, $f^\ast\omega_0|E^\pm=\lambda^{\pm 1}\cdot\omega_0|E^\pm$.

With this $\omega_0$ in hand, one can verify the following.

\begin{enumerate}
  \addtocounter{enumi}{1}
  \item The dilating metric $\omega_0$ above is a flat metric on its domain $X\setminus E$, demonstrating that our $X$ is a modification of a torus quotient. (See \S{\ref{sec:proof-step-2}} for details.)
\end{enumerate}

Consequently, the holomorphic automorphism in question gives rise to a holomorphic automorphism on a torus quotient. By this, one proves that $(X,f)$ is a Kummer example.

\subsubsection{Example: dilating metric on a Hilbert Scheme of a Kummer Surface}
For an example of the `dilating' metric $\omega_0$, consider the Hilbert scheme example $K_1(T)^{[n]}$ above. For the toral map $f_T\colon T\to T$, one can use its eigendirections $E^+$ and $E^-$ of $f_T$ to construct a `base' metric $\omega_b$ on $T$ that enjoys $f_T^\ast\omega_b|E^\pm=\lambda^{\pm 1}\cdot\omega_b|E^\pm$.

Let $\pi_i\colon T^n\to T$ be the projection onto the $i$-th factor. The `product' of the base metric, $\omega_b^{\boxtimes n}=\pi_1^\ast\omega_b+\cdots+\pi_n^\ast\omega_b$ on $T^n$, essentially constitutes the $\omega_0$ of interest. That is, if we consider the normalization $K_1(T)^{[n]}\to T^n/\Gamma$ and send $\omega_b^{\boxtimes n}$ on $T^n/\Gamma$ to $K_1(T)^{[n]}$, this results the desired $\omega_0$.

\subsubsection{Constructing the metric $\omega_0$}
\label{sec:proof-step-1}

As a preliminary step, recall the results of Gromov and Yomdin \cite{Yomdin}\cite{Gromov}, which give that the topological entropy is realized as the spectral radius of the induced map ${f}^\ast\colon H^\bullet(X)\to H^\bullet(X)$ on the cohomology ring. Specifically, a result by Oguiso \cite[Theorem 1]{oguiso2009} shows that the topological entropy $h_{\mathrm{top}}({f})$ of ${f}$ equals $nh$, where $h$ is the log of the spectral radius of ${f}^\ast\colon H^{1,1}(X)\to H^{1,1}(X)$ on the $(1,1)$-classes. This yields the eigenvectors $[\eta_+],[\eta_-]\in H^{1,1}(X)$ of ${f}^\ast$ with eigenvalues $e^h,e^{-h}$ respectively. Moreover, the sum of these eigenvectors $[\eta_+]+[\eta_-]$ can be chosen to be a big and nef class.

According to Collins and Tosatti \cite[Theorem 1.6]{CT15}, the big and nef class $[\eta_+]+[\eta_-]$ allows us to define a metric $\omega_0$ defined outside of the null locus $E\subset X$ of $[\eta_+]+[\eta_-]$. This metric is ``cohomologous'' to the class, whose sense is precised in the cited theorem or Lemma \ref{lem:01}. By the assumption $\mu=\mathrm{vol}$, we can use Jensen's inequality to obtain the desired dilating property $f^\ast\omega_0|E^\pm=e^{\pm h}\cdot\omega_0|E^\pm$. The same rates apply backwards in time, with the roles of $E^+$ and $E^-$ exchanged.

For readers who are familiar with Filip and Tosatti's work \cite{FT18}, this construction yields a higher dimensional generalization of their Ricci-flat (orbifold) metrics $(\omega_t)_{t\in\R}$ \cite[\S{2.1.7}]{FT18}. The metric $\omega_0$ at $t=0$ takes the same role in our $\omega_0$ sketched above. The trick of Jensen's inequality to show the uniform dilation rates is also parallel to that of \cite[Proposition 3.1.4]{FT18}.

\subsubsection{Flatness of $\omega_0$ and Verifying a Kummer Example}
\label{sec:proof-step-2}

The dilating property $f^\ast\omega_0|E^\pm=e^{\pm h}\cdot\omega_0|E^\pm$ not only yields that $E^\pm$ define smooth foliations, but also \emph{holomorphic} foliations. A flatness result by Benoist, Foulon, and Labourie \cite[Lemme 2.2.3(b)]{BFL92}, together with the holomorphic symplectic form, shows that $\omega_0$ defines a flat metric on the complement $X\setminus E$ of the null locus $E$ of $[\eta_+]+[\eta_-]$.

The final step involves contracting the locus $E\subset X$ to obtain a normal variety $Y$ with canonical singularities, following a result by Boucksom, Cacciola, and Lopez \cite[Theorem A]{BCL14}. This is where we use the hypothesis that $X$ is projective. 

As $X\setminus E$ is flat, $Y$ is also flat on its regular locus. Such normal varieties are known to be a torus quotient, as demonstrated by Greb, Kebekus, and Peternell \cite[Corollary 1.16]{GKP}. (This result can be enhanced, as in \cite[Theorem D]{claudon2020kahler}, for a normal space $Y$ whose regular locus is flat.)

For readers who are familiar with Filip and Tosatti's work \cite{FT18}, this process can be understood as a detour in the construction of their contraction map $\nu\colon X\to Y$ \cite[Proposition 2.1.5(ii)]{FT18}, as well as the construction of the dynamics $(Y,T_Y)$ and the metric $\omega_0$ endowed on $Y$. In the higher-dimensional case, the contraction map $X\dashrightarrow Y$ is a birational map instead, so we avoid doing constructions directly on $Y$ but rather focus on the complement $X\setminus E$ of the null locus, as the complement is isomorphic to the regular locus of $Y$.

We also remark that the proof of the flatness using a result by Benoist, Foulon, and Labourie (see Lemma \ref{lem:20}) gives a simplified proof of the flatness result discussed in \cite[Proposition 3.2.1]{FT18}.

\subsection{The Hypothesis $\mu=\mathrm{vol}$}

For the measure of maximal entropy $\mu$, in Theorem \ref{lem:00} we assumed that $\mu$ \emph{equals} to the volume measure. This contrasts with previous results of Cantat and Dupont \cite[\S{1.3} Main Theorem]{CantatDupont2020} or Filip and Tosatti \cite[Theorem 1.1.1]{FT18}, where they merely assumed $\mu\ll\mathrm{vol}$.

To explain why we see this difference, we first note the ``Green $(1,1)$-currents'' $S^\pm$ for the system $(X,f)$, as presented in Proposition \ref{lem:ergodicity} below. In some explicit examples like generalized Kummer varieties, one can compute these currents and find that \emph{restricting the Green currents $S^+$ and $S^-$ along the unstable and stable manifolds, respectively, yields metric forms on unstable and stable manifolds}.

Despite our attempt, we were not able to establish this property as a general fact of Green $(1,1)$-currents, under the assumption that $\mu\ll\mathrm{vol}$. On the other hand, once such a property is established, the Hopf arguments as in \cite[\S\S{4--5}]{CantatDupont2020} \cite[\S{5}]{FT18} come into play. Such arguments lead to the conclusion that $(X,f)$ is \emph{uniformly} hyperbolic. Consequently, following standard arguments (see, for instance, \cite[Theorem 4.2.1]{FT18}), it can be verified that $\mu\ll\mathrm{vol}$ implies $\mu=\mathrm{vol}$, hence replacing the hypothesis $\mu=\mathrm{vol}$ to $\mu\ll\mathrm{vol}$.

\subsection{Contents of the Paper} 
Section \ref{sec:prelim} is placed to introduce various notations and structures that we will recall from the hyperk\"ahler dynamics $(X,{f})$ of interest.

Section \ref{sec:analysis-on-approximate-metrics} discusses some properties of the approximate metrics (cf. Lemma \ref{lem:01}), to derive the uniform expansion and contraction observed for $\omega_0$ (cf. Corollary \ref{lem:14})

Section \ref{sec:stable-unstable-distributions} is devoted to the study of stable and unstable manifolds, especially focused on that they are holomorphic. The section ends with a proof that the metric $\omega_0$ gives a flat metric.

Section \ref{sec:upshots-of-flatness} explains how flatness readily gives a proof of Theorem \ref{lem:00}, the main theorem.


\subsection*{Acknowledgements}
The author would like to thank Simion Filip, Alex Eskin, Sang-hyun Kim, Seok Hyeong Lee, Junekey Jeon, and Hongtaek Jung for useful discussions, Valentino Tosatti for significant comments on an earlier draft. The author also want to thank for Serge Cantat, Beno\^it Claudon, and Christophe Dupont for telling extensive feedbacks, surrounding backgrounds, and possible improvements about this work. 
This material is based upon work supported by the National Science Foundation under Grant No. DMS-2005470 and DMS-2305394.

\section{Preliminaries on Hyperk\"ahler Manifolds and their Dynamical Structures}
\label{sec:prelim}

Hyperk\"ahler manifolds are rich in structures, which lead us to introduce various notations and list basic facts that are required to study them. This section is devoted to that purpose.

\subsection{Hyperk\"ahler Structures}
\label{sec:hyperkahler-structures}

Let $(X,\omega,\Omega)$ be a (projective) hyperk\"ahler manifold, whose underlying manifold $X$ has $\dim_\C X=2n$ and is simply connected. Here $\omega$ is a fixed hyperk\"ahler metric on $X$, and $\Omega$ is an everywhere nondegenerate holomorphic 2-form that $X$ should have. We will call $\Omega$ as a \emph{holomorphic symplectic form} on $X$. That $\omega$ is hyperk\"ahler may be understood that the tensor $\Omega$ is flat with respect to $\omega$. For $n=1$, a hyperk\"ahler manifold is nothing but a K3 surface.

This $\Omega$ generates the $(2,0)$-Hodge group: $H^{2,0}(X,\C)=\C.\Omega$ \cite[Proposition 23.3]{ghj01}. Moreover, we declare the \emph{volume form} $\mathrm{vol}=(\Omega\wedge\overline{\Omega})^n$ associated to the holomorphic symplectic form; we normalize $\Omega$ so that $\mathrm{vol}(X)=1$. This volume form is same as that of the Riemannian geometry on $X$: that is, $\omega^{2n}=c\cdot\mathrm{vol}$ for some constant $c>0$. Later, we will impose a normalization condition \eqref{eqn:eigenclass-normalization} for this $c$.

A hyperk\"ahler manifold has the natural quadratic form on $H^2$ that generalizes the intersection form of K3 surfaces. This is called the Beauville--Bogomolov--Fujiki quadratic form $q$ \cite[Definition 22.8]{ghj01} and is defined as follows. Suppose $\alpha\in H^2(X,\C)=H^{2,0}(X)\oplus H^{1,1}(X)\oplus H^{0,2}(X)$ is decomposed as $\alpha=c_1\Omega+\beta+c_2\overline{\Omega}$, where $c_1,c_2\in\C$ and $\beta\in H^{1,1}(X)$. Then the number $q(\alpha)$ is defined as
\begin{equation}
  \label{eqn:1.1}
  q(\alpha):=c_1c_2+\frac{n}{2}\int_X\beta^2(\Omega\wedge\overline{\Omega})^{n-1}.
\end{equation}
We will abuse the notation and denote $q(\alpha,\alpha')$ for the symmetric bilinear form from the quadratic form $q$. The form $q$, if restricted to $H^{1,1}(X,\R)$, has the signature $(1,h^{1,1}-1)$ \cite[Corollary 23.11]{ghj01}. One may view this as the Hodge Index Theorem for hyperk\"ahler manifolds.

The Beauville--Bogomolov--Fujiki quadratic form $q$, together with the Beauville--Fujiki relation \cite[Proposition 23.14]{ghj01}, gives the following formula. Given a Ricci-flat K\"ahler metric $\omega'$, we have
\begin{equation}
  \label{eqn:1.2} 
  (\omega')^{2n}=q(\omega')^n\binom{2n}{n}\cdot(\Omega\wedge \overline{\Omega})^n
\end{equation}
as differential forms.

\subsection{Invariance}

Given a hyperk\"ahler manifold $(X,\omega,\Omega)$, let ${f}\colon X\to X$ be a holomorphic automorphism that has positive topological entropy $h_{\mathrm{top}}({f})>0$. Various structures of $X$ are preserved under ${f}$.

Obviously, as a holomorphic map, ${f}$ preserves the complex structure $I$. Therefore we preserve the holomorphic sheaves $\Omega^p_X$, Hodge groups $H^{p,q}(X)$, etc. by ${f}$. We also note that ${f}$ preserves the K\"ahler cone.

Although the hyperk\"ahler metric $\omega$ is seldom invariant under ${f}$, the holomorphic symplectic form $\Omega$ is almost invariant, in the following sense.

\begin{lemma}
  \label{lem:02}
  There exists a constant $k_{f}$, with absolute value 1, such that 
  ${f}^\ast\Omega=k_{f}\Omega$.
\end{lemma}
\begin{proof}
  By Dolbault isomorphism, we can identify the group $H^{2,0}(X)=\C.\Omega$ with $H^0(X,\Omega_X^2)=\C.\Omega$. 
  Therefore any holomorphic section of 
  the vector bundle $\Omega_X^2$ on $X$ is proportional to 
  $\Omega\colon X\to\Omega_X^2$.

  Now ${f}$ induces another holomorphic section of $\Omega_X^2$, by 
  ${f}^\ast\Omega\colon X\xrightarrow{{f}}X\xrightarrow{\Omega}\Omega_X^2$. 
  This has to be proportional to $\Omega\colon X\to\Omega_X^2$, thus we have 
  ${f}^\ast\Omega=k_{f}\Omega$ for some $k_{f}\in\C$.
  Because ${f}$ preserves the volume $(\Omega\wedge\overline{\Omega})^n$,
  $|k_{f}|=1$ follows.
\end{proof}

Now recall that Beauville--Bogomolov--Fujiki form is defined only with the form $\Omega$. Because of that, it is natural to guess the

\begin{lemma}
  \label{lem:03}
  The Beauville--Bogomolov--Fujiki form $q$ is preserved under ${f}$. That is, 
  $q({f}^\ast\alpha)=q(\alpha)$ for any 2-class $\alpha\in H^2(X)$.
\end{lemma}
\begin{proof}
  First we note that ${f}^\ast(\Omega\overline{\Omega})=|k_{f}|^2\Omega\overline{\Omega}=\Omega\overline{\Omega}$. Thus ${f}^\ast$ acts on $H^{4n}(X)=\C.(\Omega\overline{\Omega})^n$ trivially, which gives the identity $\int_X{f}^\ast\gamma=\int_X\gamma$ for any $4n$-form $\gamma$.

  Write $\alpha=c_1\Omega+\beta+c_2\overline{\Omega}$, where $c_1,c_2\in\C$ 
  and $\beta\in H^{1,1}(X)$. Then
  ${f}^\ast\alpha=c_1({f}^\ast\Omega)+({f}^\ast\beta)+c_2({f}^\ast\overline{\Omega})=c_1 k_{f}\Omega+({f}^\ast\beta)+c_2\overline{k}_{f}\overline{\Omega}$. By \eqref{eqn:1.1}, we compute
  \begin{align*}
    q({f}^\ast\alpha) &= (c_1k_f)(c_2\overline{k}_f)
        +\frac{n}{2}\int_X({f}^\ast\beta)^2(\Omega\overline{\Omega})^{n-1} \\
    &= c_1c_2|k_f|^2 
        +\frac{n}{2}\int_X({f}^\ast\beta)^2({f}^\ast(\Omega\overline{\Omega}))^{n-1} \\
    &= c_1c_2
        +\frac{n}{2}\int_X{f}^\ast[\beta^2(\Omega\overline{\Omega})^{n-1}] \\
    &= c_1c_2+\frac{n}{2}\int_X\beta^2(\Omega\overline{\Omega})^{n-1}=q(\alpha),
  \end{align*}
  and get the equation demanded.
\end{proof}

\subsection{The Eigenclasses}
\label{subsec:eigenclass}

Lemma \ref{lem:03} readily implies that ${f}^\ast$ is an \emph{isometry} of the hyperbolic space $\Hh_X$, the connected component of the hyperboloid $\{x\in H^{1,1}(X,\R) \mid q(x)=1\}$ that contains a K\"ahler class. We equip the metric $q$ to make $\Hh_X$ a Riemannian manifold.

The isometry ${f}^\ast$ is loxodromic in the sense of \cite[\S{2.3.2}]{Cantat-lecturenote}, if $h_{\mathrm{top}}(f)>0$. This fact follows from the following estimate of the first dynamical degree, by Oguiso.

\begin{extTheorem}[{\cite[Thm. 1.1]{oguiso2009}}]
  Let $X$ be a hyperk\"ahler manifold of dimension $2n$ and ${f}$ be a holomorphic automorphism of $X$. If $d_k({f})$ is the $k$-th dynamical degree of ${f}$, i.e., the spectral radius of ${f}^\ast$ on $H^{k,k}(X,\R)$, then we have $d_{2n-k}({f})=d_k({f})=d_1({f})^k$ for $0\leq k\leq n$. Moreover, the topological entropy $h_{\mathrm{top}}({f})$ is $n\log d_1({f})$.
\end{extTheorem}

Thus if $h_{\mathrm{top}}({f})>0$, then $\log d_1({f})>0$ as well. Define the numbers
\begin{equation}
  \label{eqn:h-and-lambda}
  h:=\log d_1({f}),\quad\lambda :=d_1({f}).
\end{equation}
(A caveat is that the entropy is $h_{\mathrm{top}}({f})=nh$, not $h$ as the notation may suggest.) Thus we have an exponential estimate $\|({f}^\ast)^N\|=O(\lambda^N)$ on $\Hh_X$, so ${f}^\ast$ must be loxodromic.

Because ${f}^\ast$ is loxodromic, we can find eigenclasses $[\eta_+],[\eta_-]\in H^{1,1}(X,\R)$ of real single eigenvalues $\lambda,\lambda^{-1}$ respectively (cf. \cite[Theorem 1]{oguiso2007}). By the prescribed metric $\omega$ in \S{\ref{sec:hyperkahler-structures}}, we can specifically let the eigenclasses $[\eta_\pm]$ be limits of $\lambda^{-n}(f^\ast)^{\pm n}[\omega]$, $n\to\infty$, respecting $\pm$. Such limits are well-defined (see \cite[\S{2.3.2}]{Cantat-lecturenote}) and define nef classes. They have the following further properties.

\begin{proposition}
  \label{lem:04}
  The eigenclasses $[\eta_+],[\eta_-]\in H^{1,1}(X,\R)$ satisfy the followings.
  \begin{enumerate}[(a)]
    \item \label{enum:lem:04:a} Isotropic: $q([\eta_+])=q([\eta_-])=0$.
    \item \label{enum:lem:04:b} Nilpotent: $[\eta_+]^{n+1}=[\eta_-]^{n+1}=0$.
    \item \label{enum:lem:04:c} Big and Nef: $[\eta_+]+[\eta_-]$ is a big and nef class.
    \item \label{enum:lem:04:d} Spectrally convergent: for any K\"ahler class $\alpha\in H^{1,1}(X,\R)$, we have
    \begin{align*}
      \lambda^{-n}({f}^\ast)^n\alpha &\to \frac{q(\alpha,[\eta_-])}{q([\eta_+],[\eta_-])}[\eta_+], \\
      \lambda^{-n}({f}^\ast)^{-n}\alpha &\to \frac{q(\alpha,[\eta_+])}{q([\eta_+],[\eta_-])}[\eta_-].
    \end{align*}
  \end{enumerate}
\end{proposition}
\begin{proof}
  The isotropy property is obvious: $q([\eta_\pm])=q({f}^\ast[\eta_\pm])=q(\lambda^{\pm 1}[\eta_\pm])=\lambda^{\pm2}q([\eta_\pm])$ and $\lambda>1$ gives the result.

  The nilpotence follows from \cite[Proposition 24.1]{ghj01}\cite[Theorem 15.1]{verbitsky1995}: for any $\alpha\in H^2(X,\C)$, $\alpha^{n+1}=0$ iff $q(\alpha)=0$.


  We claim that the nef class $[\eta_+]+[\eta_-]$ is big. Note the Beauville--Fujiki relation $q(\alpha)^n=\binom{2n}{n}^{-1}\int_X\alpha^{2n}$ \cite[Proposition 23.14]{ghj01}. If $[\eta_+]+[\eta_-]$ is not big, this implies $\int_X(\eta_++\eta_-)^{2n}=0$ and thus $q([\eta_+]+[\eta_-])=0$. But then $q([\eta_+],[\eta_-])=0$ follows, so $q(c_1[\eta_+]+c_2[\eta_-])=0$ for all $c_1,c_2\in\R$. This contradicts with that the cone $\{q=0\}$ cannot contain a linear space of dimension $>1$.

  The spectral convergences are in general the case whenever $q(\alpha)>0$: see \cite[\S{2.3.2}]{Cantat-lecturenote}.
\end{proof}

We notice that $[\eta_+]$ and $[\eta_-]$ depend on the scale of $\omega$. To avoid confusion caused by this, we impose a normalization condition
\begin{equation}
  \label{eqn:eigenclass-normalization}
  \int_X(\eta_++\eta_-)^{2n} = 1.
\end{equation}
By this, we have $q([\eta_+]+[\eta_-])=2q([\eta_+],[\eta_-])=\binom{2n}{n}^{-1/n}$.

\subsection{Null Locus and Metric Approximations}

Even though the class $[\eta_+]+[\eta_-]$ is big and nef, it is very unlikely to be K\"ahler. Because of that, we introduce the obstruction studied in \cite{CT15} for a big and nef class to be K\"ahler.

For sake of introducing an obstruction, one introduces the 
\emph{null locus} $E\subset X$ of the class $[\eta_+]+[\eta_-]$, which is the union of all 
subvarieties $V\subset X$ such that $\int_{V}(\eta_++\eta_-)^{\dim V}=0$ \cite[p.1168]{CT15}. 
By \cite[Theorem 1.1]{CT15}, this null locus is the same as the non-K\"ahler locus of $[\eta_+]+[\eta_-]$, which is a proper analytic subvariety of $X$.

\begin{proposition}
  The null locus $E$ defined above is ${f}$-invariant.
\end{proposition}
\begin{proof}
  We check that it satisfies ${f}(E)\subset E$. 
Recall that $E$ is defined as the union of subvarieties $V$ in which 
$\int_V(\eta_++\eta_-)^{\dim V}=0$. Because $[\eta_+]$ and $[\eta_-]$ are nef classes, thanks to \cite[Theorem 4.5]{DP04} and approximation of $[\eta_\pm]$ by K\"ahler classes,
we have $\int_V(\eta_+)^a(\eta_-)^b\geq 0$ whenever $a+b=\dim V$. Hence if $\int_V(\eta_++\eta_-)^{\dim V}=0$, we have $\int_V(\eta_+)^a(\eta_-)^b=0$.
Because ${f}^\ast[\eta_\pm]=\lambda^{\pm1}[\eta_\pm]$, this implies 
$\int_{{f}(V)}(\eta_+)^a(\eta_-)^b=0$ as well, now integrating on $f(V)$. 
Collecting them we have $\int_{{f}(V)}(\eta_++\eta_-)^{\dim V}=0$, 
verifying ${f}(V)\subset E$.
\end{proof}

Although the following is an immediate application of \cite[Theorem 1.6]{CT15}, we state this as a lemma, to introduce some notations for later use.
\begin{lemma}
  \label{lem:01}
  There exists a smooth, incomplete Ricci-flat K\"ahler metric $\omega_0$ on $X\setminus E$, and a sequence $\left(\omega_k\right)$ of complete, smooth Ricci-flat \emph{hyperk\"ahler} metrics on $X$ that converges to $\omega_0$ in the following senses: (i) $[\omega_k]\to[\eta_+]+[\eta_-]$ in $H^{1,1}(X,\R)$, and (ii) $\omega_k\to\omega_0$ in $C^\infty_{\mathrm{loc}}(X\setminus E)$.

  Moreover, $\omega_k$'s may be set to have the unit volume. That is, $\omega_k^{2n}=\mathrm{vol}$.
\end{lemma}
\begin{proof}
  \cite[Theorem 1.6]{CT15} tells the existence of a smooth, incomplete Ricci-flat K\"ahler metric $\omega_0$ on $X\setminus E$ in which, for any sequence $[\alpha_k]\to[\eta_+]+[\eta_-]$ in $H^{1,1}(X)$, and the Ricci-flat metric $\omega_k\in[\alpha_k]$, $\omega_k$'s converge to $\omega_0$ in $C^\infty_{\mathrm{loc}}(X\setminus E)$ topology.

  For hyperk\"ahler $X$, each K\"ahler class $[\alpha_k]$ contains a unique hyperk\"ahler metric in it \cite[Theorem 23.5]{ghj01}. Such a metric is necessarily Ricci-flat \cite[Proposition 4.5]{ghj01}. By the uniqueness of Ricci-flat metric in a K\"ahler class, $\omega_k$ must be hyperk\"ahler as well.

  Because $[\alpha_k]$'s are converging to $[\eta_+]+[\eta_-]$, a big and nef class, 
  the volume $[\alpha_k]^{2n}$ by $\omega_k$, also converges to a 
  positive number. Thus one can normalize $\omega_k$'s so that they 
  have unit volume, i.e., $[\alpha_k]^{2n}=1$ for all $k$. 
  Consequently, as each $\omega_k$ is Ricci-flat, we have 
  $\omega_k^{2n}=(\Omega\wedge\overline{\Omega})^{2n}=\mathrm{vol}$, due to 
  \eqref{eqn:1.2}.
  (Note that this normalization is compatible with \eqref{eqn:eigenclass-normalization}.)
\end{proof}

\subsection{Ergodicity and Lyapunov Exponents}

As a consequence of bigness (Proposition \ref{lem:04}) of $[\eta_+]+[\eta_-]$, together with results of Dinh--Sibony \cite{DS04}\cite{DS10} and De Th\'elin--Dinh \cite{DTD12}, we have the following

\begin{proposition}
  \label{lem:ergodicity}
  \begin{enumerate}[(a)]
    \item There exists closed positive $(1,1)$-currents $S^+$ and $S^-$ with H\"older continuous potentials, that are in classes $[\eta_+]$ and $[\eta_-]$ respectively, and satisfy $f^\ast S^+=\lambda S^+$ and ${f}^\ast S^-=\lambda^{-1}S^-$.
    \item The wedge $(S^+)^n\wedge(S^-)^n$ in the sense of Bedford--Taylor \cite{BT82} is $\binom{2n}{n}^{-1}$ times the unique measure $\mu$ of maximal entropy, which is mixing.
  \end{enumerate}
\end{proposition}

The currents $S^\pm$ mentioned above are called \emph{Green currents} of order 1, according to \cite[\S{4.2}]{DS10}. We will call $S^+$ an \emph{unstable Green $(1,1)$-current} and $S^-$ a \emph{stable Green $(1,1)$-current}.

\begin{proof}
  To show (a), first find a positive closed $(1,1)$-current $S_1^+$ satisfying ${f}^\ast S_1^+=\lambda S_1^+$, by \cite[Corollary 3.5]{DS04}. The potential of $S_1^+$ is H\"older continuous and the class $[S_1^+]$ belongs to the closure of the K\"ahler cone; in particular, the class is nef. Because $f^\ast[S_1^+]=\lambda[S_1^+]$, the class $[S_1^+]$ is positively proportional to $[\eta_+]$. So by rescaling $S_1^+$ we have the demanded $S^+$. Same construction applies for $S^-$, by considering ${f}^{-1}$. This shows (a).


  Before showing (b), we study some nature of $T^+=(S^+)^n$ and $T^-=(S^-)^n$. 
  Because potentials of $S^\pm$ are locally bounded, the Bedford--Taylor theory applies and defines $(S^\pm)^n$ as closed positive $(n,n)$-currents. Hence $T^+=(S^+)^n$ and $T^-=(S^-)^n$ are well-defined. Note that each current is in class $[\eta_+]^n$ and $[\eta_-]^n$, respectively.
  
  It turns out that $T^+$ and $T^-$ are Green currents of $f$ and $f^{-1}$ respectively, of order $n$. 
  To see so, let $V_+=\R.[\eta_+]^n\subset H^{n,n}(X,\R)$. Then ${f}^\ast$ acts as a multiplication by $\lambda^n$ on $V_+$, and we have $\lambda^n=d_n(f)>\lambda^{n-1}=d_{n-1}(f)$. Thus the hypotheses of \cite[Theorem 4.3.1]{DS10} are met, and $T^+$ is the unique closed positive $(n,n)$-current in the class $[\eta_+]^n$. We argue likewise with $V_-=\R.[\eta_-]^n\subset H^{n,n}(X,\R)$ and $(f^{-1})^\ast$, to have that $T^-$ is the unique closed positive $(n,n)$-current in the class $[\eta_-]^n$.

  We claim that $T^+\wedge T^-$ is a positive nonzero measure. To have so, note first that $T^+\wedge T^-$ is cohomologous to $\eta_+^n\wedge\eta_-^n$ (as currents), where $\eta_\pm$ are any smooth representative of the classes $[\eta_\pm]$. Then we have
  \begin{align*}
    \int_X T^+\wedge T^- &= \int_X\eta_+^n\wedge\eta_-^n \\
    &= \binom{2n}{n}^{-1}\int_X(\eta_++\eta_-)^{2n} = \binom{2n}{n}^{-1},
  \end{align*}
  by the normalization \eqref{eqn:eigenclass-normalization}. As $T^+\wedge T^-$ is already a positive $(2n,2n)$-current, this suffices to show that $T^+\wedge T^-$ is a positive nonzero measure. A corollary is, by \cite[Proposition 4.4.1]{DS10}, that the eigenvalues $\lambda^{\pm n}$ of $f^\ast$ on $H^{n,n}(X,\C)$ have multiplicities 1.

  Now we are ready to prove (b). 
  We can apply \cite[Theorem 1.2]{DTD12} here. To see so, we remark that the proof of the theorem only requires $f$ to be `simple on $\bigoplus_{p=0}^{2n}H^{p,p}(X,\C)$,' i.e., admits a unique, multiplicity 1 eigenvalue of absolute value $d_n(f)$, at the subring $\bigoplus_{p=0}^{2n}H^{p,p}(X,\C)$ of the cohomology ring. Because $f^\ast$ on other $H^{p,p}$ groups have spectral radius strictly less than $d_n(f)$, we see that the multiplicity of $d_n(f)$ on the subring is precisely that on the group $H^{n,n}(X,\C)$. This was observed to be multiplicity 1 above.

  Thus the theorem applies, and the Green measure $\mu=\binom{2n}{n}T^+\wedge T^-$ is the unique invariant \emph{probability} measure of maximal entropy. By \cite[Theorem 4.4.2]{DS10}, we furthermore know that $\mu$ is mixing.
\end{proof}

We also claim that, if the Green measure $\mu$ is in the volume class, i.e., absolutely continuous with repsect to the volume measure $\mathrm{vol}$ (which comes from the Riemannian volume form $c^{-1}\omega^{2n}$), then the Lyapunov exponents are very simple. Recall the quantity $h=\log d_1({f})=\frac1n h_{\mathrm{top}}({f})$ defined in \eqref{eqn:h-and-lambda}.

\begin{lemma}
  \label{lem:09}
  If the Green measure $\mu=(S^+)^n\wedge(S^-)^n$ is in the volume class, the Lyapunov exponents are $\pm h/2$ with multiplicity $2n$ each.
\end{lemma}
\begin{proof}
  Let $\chi_1\geq\cdots\geq\chi_{2n}$ be the Lyapunov exponents of $\mu$, listed with multiplicities, for the cocyle $Df$ on the complexified tangent bundle $T_\C X$ (cf. \cite[Theorem 2.2.6]{Fil17}\cite[Theorem 1.6]{ruelle1979}). By ergodicity of $\mu$, they are $\mu$-a.e. constant. 

  Thanks to Oguiso \cite[Theorem 1.1]{oguiso2009}, we have an increasing-decreasing relation
  \[1=d_0({f})<d_1({f})<\cdots<d_{n-1}({f})<d_n({f})>d_{n+1}({f})>\cdots>d_{2n}({f})=1\]
  of dynamical degrees. The bounds of Lyapunov exponents by de Th\'elin \cite[Corollaire 3]{The08} then yields,
  \begin{equation}\label{eqn:lem:09-2}\chi_1\geq\cdots\geq\chi_{n}\geq\frac{1}{2}\log\frac{d_n({f})}{d_{n-1}({f})}=\frac{h}{2}(>0),\end{equation}
  and
  \begin{equation}\label{eqn:lem:09-3}\chi_{2n}\leq\cdots\leq\chi_{n+1}\leq\frac12\log\frac{d_{n+1}(f)}{d_n(f)}=-\frac{h}{2}(<0).\end{equation}
  In particular, the first $n$ exponents are nonnegative and the last $n$ exponents are nonpositive.

  Because the measure $\mu$ is in the volume class, by Ledrappier--Young formula \cite[Corollary G]{LY85II}, the entropy $h_\mu(f)$ equals
  \begin{equation}\label{eqn:lem:09-1}h_\mu({f})=nh=2\cdot(\chi_1+\cdots+\chi_{n}).\end{equation}
  Here, $h_\mu(f)=nh$ follows from $h_\mu(f)=h_{\mathrm{top}}(f)=nh$.   
  Combining \eqref{eqn:lem:09-1} and \eqref{eqn:lem:09-2}, we have $\chi_1=\cdots=\chi_n=h/2$. 
  
  Likewise, focusing on $h_\mu(f^{-1})$, the Ledrappier--Young formula yields
  \begin{equation}\label{eqn:lem:09-4}h_\mu({f}^{-1})=nh=2\cdot(-\chi_{n+1}-\cdots-\chi_{2n}).\end{equation}
  (To see why $-\chi_{n+1},\ldots,-\chi_{2n}$ are suitable for $f^{-1}$, see \cite[\S{2.2.10}]{Fil17}.) Therefore, combining \eqref{eqn:lem:09-3} and \eqref{eqn:lem:09-4}, we have $\chi_{n+1}=\cdots=\chi_{2n}=-h/2$.
\end{proof}

\section{Analysis on Approximate Metrics}
\label{sec:analysis-on-approximate-metrics}

In this section, we aim to study how the limit metric $\omega_0$ behaves under the dynamics. To elaborate, we show the following (Corollary \ref{lem:14}). For each $x\in X\setminus E$, we have $n$-dimensional complex subspaces $E^\pm_{N,x}\subset T^{1,0}_xX$ (where $T^{1,0}X$ is the holomorphic tangent space) so that, applying $(f^N)^\ast$ to $\omega_{0,x}|E^\pm_{N,x}$, we rescale it by $\lambda^N=e^{Nh}$.

\subsection{Local Setups}
\label{subsec:local-setups}

As claimed in Lemma \ref{lem:01}, 
we pick up hyperk\"ahler metrics $\omega_k$ that converge to $\omega_0$ in 
$C^\infty_{\mathrm{loc}}(X\setminus E)$ topology.

Fix $N\in\Z_{>0}$. Along the dynamics, we define the quantitites 
$\left(\sigma_i^{(k)}(x,Nh)\right)_{i=1}^{2n}$, to be called 
\emph{log-singular values of $({f}^N)^\ast\omega_k$ relative to $\omega_k$}, 
as follows. At a point $x\in X$, one can write 
\begin{equation}
  \label{eqn:3.2} 
  \omega_k=\frac{\sqrt{-1}}{2}\sum_{i=1}^{2n}dz_i\wedge d\overline{z}_i,
\end{equation}
with an appropriate holomorphic coordinate $(z_1,\cdots,z_{2n})$ of $X$ at $x$. 
Moreover, for
\begin{equation}
  \label{eqn:3.3}
  h_{ij}^{(k)}=({f}^N)^\ast\omega_k(\frac{\partial}{\partial z_i},\frac{\partial}{\partial z_j}),
\end{equation}
we have a self-adjoint matrix $\left(h_{ij}^{(k)}\right)$. 
Thus one can adjust the holomorphic basis vectors
$\dfrac{\partial}{\partial z_1},\cdots,\dfrac{\partial}{\partial z_{2n}}$ 
unitarily, to keep the form \eqref{eqn:3.2} of $\omega_k$ yet make the matrix
$\left(h_{ij}^{(k)}\right)$ diagonal. 
This enables us to write $({f}^N)^\ast\omega_k$ at $x\in X$ as
\begin{equation}
  \label{eqn:3.4} 
  ({f}^N)^\ast\omega_k=\frac{\sqrt{-1}}{2}\sum_{i=1}^{2n}\exp\left(\sigma_i^{(k)}(x,Nh)\right)\, dz_i\wedge d\overline{z}_i.
\end{equation}
We moreover choose the base index $i$ so that the numbers 
$\sigma_1^{(k)},\cdots,\sigma_{2n}^{(k)}$ are in the decreasing order, i.e.,
\[
  \sigma_1^{(k)}(x,Nh)\geq\sigma_2^{(k)}(x,Nh)\geq\cdots\geq\sigma_{2n}^{(k)}(x,Nh).
\]
These quantities exhibit the following symmetry.
The assumption that $\omega_k$ is hyperk\"ahler is crucial here.

\begin{lemma}
  \label{lem:05}
  Log-singular values $\sigma_i^{(k)}$'s exhibit symmetry at 0. 
  That is, for all $x$ and $i=1,2,\cdots,n$,
  \[\sigma_i^{(k)}(x,Nh)+\sigma_{2n+1-i}^{(k)}(x,Nh)=0.\]
\end{lemma}
\begin{proof}
  Replacing ${f}$ to ${f}^N$ if necessary, we may assume that $N=1$. 
  Also, for notational simplicity, denote $\omega:=\omega_k$. We denote the transpose of $A=(a_{ij})$ as $A^\top=(a_{ji})$ and the conjugate transpose as $A^\dagger=(\overline{a_{ji}})$.
  
  As $\omega$ is a hyperk\"ahler metric, for each point $x\in X$ one has a 
  holomorphic coordinate $(z_1,\cdots,z_{2n})$ that enables representations (validity guaranteed only at $x$)
  \begin{align*}
    \omega &= \sum_{i=1}^{2n}dz_i\wedge d\overline{z}_i, \\
    \Omega &= \sum_{\mu=1}^n dz_\mu\wedge dz_{n+\mu}.
  \end{align*}
  
  Denote $(w_1,\cdots,w_{2n})$ for another holomorphic coordinate near ${f}(x)$ 
  with analogous expressions of $\omega$ and $\Omega$ at ${f}(x)$. 
  By this coordinate, describe the map $D_xf\colon T_xX\to T_{{f}(x)}X$ as a matrix 
  $A=\left(a_{ij}\right)$, where $a_{ij}$'s are determined by the relation
  \[D_x{f}\left(\frac{\partial}{\partial z_i}\right)=\sum_{j=1}^{2n}a_{ij}\frac{\partial}{\partial w_j}.\]

  This $A$ satisfies $AA^\dagger\in\mathrm{Sp}(2n,\C)$. 
  Indeed by Lemma \ref{lem:02}, ${f}^\ast\Omega=k_{f}\Omega$, implies
  \[\sum_{k,\ell=1}^{2n}a_{ik}\Omega_{k\ell}a_{j\ell}=k_{f}\Omega_{ij}\]
  (where $\Omega_{ij}=\Omega\left(\dfrac{\partial}{\partial z_i},\dfrac{\partial}{\partial z_j}\right)$;
  in other words, 
  $(\Omega_{ij})=\begin{bmatrix} 0 & I_n \\ -I_n & 0 \end{bmatrix}$). This yields $A\Omega A^\top=k_{f}\Omega$. 
  Thus $k_{f}^{-1/2}A^\top\in\mathrm{Sp}(2n,\C)$, 
  which implies $AA^\dagger\in\mathrm{Sp}(2n,\C)$.

  We describe how ${f}^\ast\omega$ is represented at $x$. Since
  \begin{align*}
    {f}^\ast\omega\left(\frac{\partial}{\partial z_i},\frac{\partial}{\partial z_j}\right) &= 
    \sum_{k,\ell=1}^{2n}\omega\left(a_{ik}\frac{\partial}{\partial w_k},a_{j\ell}\frac{\partial}{\partial w_\ell}\right) \\
    &= \sum_{k,\ell=1}^{2n}a_{ik}\overline{a_{j\ell}}\omega\left(\frac{\partial}{\partial w_k},\frac{\partial}{\partial w_\ell}\right) \\
    &= \sum_{k=1}^{2n}a_{ik}\overline{a_{jk}} = (AA^\dagger)_{ij},
  \end{align*}
  we have
  ${f}^\ast\omega=\frac{\sqrt{-1}}{2}\sum_{i,j=1}^{2n}(AA^\dagger)_{ij}\, dz_i\wedge d\overline{z}_j$
  at $x$. Thus the numbers $\sigma_i^{(k)}$ are log of eigenvalues of $AA^\dagger$. 
  A fact in symplectic matrices \cite[Problem 22]{dGM2011} yields that 
  a self-adjoint positive-definite symplectic matrix like $AA^\dagger$ 
  has eigenvalues that are (multiplicatively) symmetric at 1.
  Taking log to that symmetry relation, 
  we get our desired symmetry of $\sigma_1^{(k)},\cdots,\sigma^{(k)}_{2n}$.
\end{proof}

\subsection{Local computations}

Based on the setups established in the previous section \ref{subsec:local-setups}, we now respectively compute the forms $\omega_k^{2n}$ and $\omega_k^{2n-1}\wedge({f}^N)^\ast\omega_k$ at $x$, and compare them to establish a
\begin{proposition}
  \label{lem:06}
  As differential forms,
  \[\omega_k^{2n-1}\wedge({f}^N)^\ast\omega_k=\left[\frac{1}{n}\sum_{j=1}^n\cosh\left(\sigma_j^{(k)}(x,Nh)\right)\right]\omega_k^{2n}.\]
\end{proposition}
\begin{proof} We use the coordinates as in \eqref{eqn:3.2} and \eqref{eqn:3.4}. Compute
\begin{align} 
  \omega_k^{2n}&=\left(\frac{\sqrt{-1}}{2}\right)^{2n}(2n)!(dz_1\wedge d\overline{z}_1\wedge\cdots\wedge dz_{2n}\wedge d\overline{z}_{2n}),\label{eqn:3.6} \\
  \omega_k^{2n-1}\wedge({f}^N)^\ast\omega_k &= \left(\frac{\sqrt{-1}}{2}\right)^{2n}\left(\sum_{i=1}^{2n}dz_i\wedge d\overline{z}_i\right)^{2n-1}\wedge\left(\sum_{j=1}^{2n}e^{\sigma_j^{(k)}(x,Nh)}dz_j\wedge d\overline{z}_j\right) \nonumber \\
  &= \left(\frac{\sqrt{-1}}{2}\right)^{2n}\left[\sum_{j=1}^{2n}e^{\sigma_j^{(k)}(x,Nh)}\right](2n-1)!(dz_1\wedge d\overline{z}_1\wedge\cdots\wedge dz_{2n}\wedge d\overline{z}_{2n}). \label{eqn:3.7} 
  \intertext{By \eqref{eqn:3.6},}
  &= \left[\sum_{j=1}^{2n}e^{\sigma_j^{(k)}(x,Nh)}\right]\frac{(2n-1)!}{(2n)!}\omega_k^{2n}, \label{eqn:3.8} 
  \intertext{and by Lemma \ref{lem:05}, we can `fold' the sum of 
  exponentials as}
  &=\left[\sum_{j=1}^n2\cosh(\sigma_j^{(k)}(x,Nh))\right]\frac{\omega_k^{2n}}{2n}, \nonumber
\end{align}
and the proposition follows.
\end{proof}

\subsection{Cohomological Analysis}

The form $\omega_k^{2n-1}\wedge({f}^N)^\ast\omega_k$ can also be understood 
cohomologically, but with some approximations. Consider the following integral, 
represented as a cup product of cohomology classes:
\[\int_X\omega_k^{2n-1}\wedge({f}^N)^\ast\omega_k = [\omega_k]^{2n-1}.({f}^N)^\ast[\omega_k].\]
Denote $[\omega_0]:=[\eta_+]+[\eta_-]$. 
Then, by Lemma \ref{lem:01}, 
we have that $[\omega_k]\to[\omega_0]$ as $k\to\infty$. 
This leads us to consider the product $[\omega_0]^{2n-1}.({f}^N)^\ast[\omega_0]$, 
which is evaluated in the following
\begin{proposition}
  \label{lem:07}
  We have the following equation in cohomology:
  \[[\omega_0]^{2n-1}.({f}^N)^\ast[\omega_0]=\cosh(Nh)\cdot[\omega_0]^{2n}.\]
\end{proposition}
\begin{proof}
  The proof is a manual computation with 
  $[\omega_0]=[\eta_+]+[\eta_-]$ and Proposition \ref{lem:04}. Start with
\begin{align}
  [\omega_0]^{2n} &= ([\eta_+]+[\eta_-])^{2n} = \sum_{k=0}^{2n}\binom{2n}{k}[\eta_+]^k[\eta_-]^{2n-k} \nonumber \\
  &= \binom{2n}{n}[\eta_+]^n[\eta_-]^n, \label{eqn:4.1} \\
  [\omega_0]^{2n-1}.({f}^N)^\ast[\omega_0] &= ([\eta_+]+[\eta_-])^{2n-1}.(e^{Nh}[\eta_+]+e^{-Nh}[\eta_-]) \nonumber \\
  &=\left[\sum_{k=0}^{2n-1}\binom{2n}{k}[\eta_+]^k[\eta_-]^{2n-1-k}\right].(e^{Nh}[\eta_+]+e^{-Nh}[\eta_-]) \nonumber \\
  &=\left[\binom{2n-1}{n-1}[\eta_+]^{n-1}[\eta_-]^n+\binom{2n-1}{n}[\eta_+]^n[\eta_-]^{n-1}\right] \nonumber\\
  &\hspace{.5\textwidth}.(e^{Nh}[\eta_+]+e^{-Nh}[\eta_-]) \nonumber \\
  &=\binom{2n-1}{n-1}e^{Nh}[\eta_+]^n[\eta_-]^n+\binom{2n-1}{n}e^{-Nh}[\eta_+]^n[\eta_-]^n, \label{eqn:4.2} 
  \intertext{and because
  $\binom{2n-1}{n-1}=\binom{2n-1}{n}=\frac{1}{2}\binom{2n}{n}$, we have}
  &=\frac{1}{2}\binom{2n}{n}(e^{Nh}+e^{-Nh})[\eta_+]^n[\eta_-]^n. \label{eqn:4.3} 
  \intertext{Equation \eqref{eqn:4.1} then applies to give}
  &=\frac{1}{2}(e^{Nh}+e^{-Nh})[\omega_0]^{2n}=\cosh(Nh)[\omega_0]^{2n}. \label{eqn:4.4}
\end{align}
\end{proof}

We discuss what does this give for $[\omega_k]$'s. Because $[\omega_k]$'s converge to $[\omega_0]$ in cohomology, we see that $[\omega_k]^{2n}$ and $[\omega_k]^{2n-1}.({f}^N)^\ast[\omega_k]$ respectively converge to $[\omega_0]^{2n}$ and $[\omega_0]^{2n-1}.({f}^N)^\ast[\omega_0]$.
That is, we have the following limit:
\begin{align*}
  \lim_{k\to\infty}[\omega_k]^{2n-1}.({f}^N)^\ast[\omega_k] - \cosh(Nh)[\omega_k]^{2n} &= [\omega_0]^{2n-1}.({f}^N)^\ast[\omega_0]-\cosh(Nh)[\omega_0]^{2n} \\
  &=0,
\end{align*}
where the last zero is a consequence of Proposition \ref{lem:07}. If this is rewritten in the integral form, and combined 
with the local computation in Proposition \ref{lem:06}, we have
\begin{equation}
  \label{eqn:4.5}
  \int_X\left[\frac{1}{n}\sum_{j=1}^n\cosh\left(\sigma_j^{(k)}(x,Nh)\right)\right]\omega_k^{2n} - \cosh(Nh)\int_X\omega_k^{2n} \xrightarrow{k\to\infty} 0.
\end{equation}
Lemma \ref{lem:01} yields $\omega_k^{2n}=\mathrm{vol}$. 
Thus \eqref{eqn:4.5} is equivalently written as
\begin{equation}
  \label{eqn:4.6}
  \int_X\left[\frac{1}{n}\sum_{j=1}^n\cosh\left(\sigma_j^{(k)}(x,Nh)\right)\right]d\mathrm{vol}(x) \xrightarrow{k\to\infty}\cosh(Nh).
\end{equation}

\subsection{Relation with Lyapunov Exponents}
\label{sec:relation-with-lyapunov-exponents}

Although log-singular values $\sigma^{(k)}_i(x,Nh)$ come from analytic interests on forms $(f^N)^\ast\omega_k$, these quantities are closely related with Lyapunov exponents $\chi_1\geq\cdots\geq\chi_{2n}$ (see Lemma \ref{lem:09}). In particular, we have the following

\begin{proposition}
  \label{lem:08}
  The log-singular values $\sigma_i^{(k)}(x,Nh)$'s in \eqref{eqn:3.4} satisfy
  \begin{equation}
    \label{eqn:5.1}
  \frac{1}{N}\sum_{j=1}^n\int_X\sigma_j^{(k)}(x,Nh)\,d\mathrm{vol}(x)\geq 2\sum_{i=1}^n\chi_i=nh.
  \end{equation}
  Consequently, we have
  \begin{equation}
    \label{eqn:5.2} 
    \int_X\left[\frac{1}{n}\sum_{j=1}^n\sigma_j^{(k)}(x,Nh)\right]\,d\mathrm{vol}(x)\geq Nh.
  \end{equation}
\end{proposition}

We start by establishing the following computational lemma to clear the potential confusion caused by the definition of log-singular values.

\begin{lemma}
  \label{lem:12}
  Endow the metric $\omega_k$ on the holomorphic tangent bundle $T^{1,0}X$. Then
  \[\log\left\|(D_x{f}^N)^{\wedge n}\right\|_{op}=\frac12\left(\sigma_1^{(k)}(x,Nh)+\cdots+\sigma_n^{(k)}(x,Nh)\right),\]
  where the wedge is taken over $\C$.
\end{lemma}
\begin{proof}
  Fix a local coordinate $(z_1,\ldots,z_{2n})$ as in \eqref{eqn:3.2} and \eqref{eqn:3.4}. Denote
  \[\partial_i:=\frac{\partial}{\partial z_i},\quad\overline{\partial}_i:=\frac{\partial}{\partial\overline{z}_i},\]
  so that $\partial_i$'s form a $\C$-basis of $T^{1,0}_xX$. Denote $\|{}\cdot{}\|_{k,j}$ the norm by the metric $(f^j)^\ast\omega_k$. One evaluates the operator norm as
  \begin{equation}
    \label{eqn:5.7}
    \|(D_x{f}^N)^{\wedge n}\|_{op} = \sup_{\substack{v\in \bigwedge^{n}T^{1,0}_{x}X \\ v\neq0}}\frac{\|(D_x{f}^N)^{\wedge n}(v)\|_{k}}{\|v\|_{k}}=\sup_{\substack{v\in \bigwedge^{n}T^{1,0}_{x}X \\ v\neq 0}}\frac{\|v\|_{{k,N}}}{\|v\|_{k,0}}.
  \end{equation}
  
  By \eqref{eqn:3.4}, we have
  \begin{equation}
    \label{eqn:5.6} 
    \left\|\partial_i\right\|_{{k,N}}=e^{\sigma_i^{(k)}(x,Nh)/2},
  \end{equation}
  and by \eqref{eqn:3.2}, we have $\|\partial_i\|_{k,0}=1$. 
  Note also that the vectors $\partial_i$'s are orthogonal at $x$, with respect to the metrics $\omega_k$ and $(f^N)^\ast\omega_k$. Because $\sigma^{(k)}_i$ is decreasing in $i$, we see that $v=\partial_1\wedge\cdots\wedge\partial_n$ maximizes the ratio in \eqref{eqn:5.7}. Hence we evaluate
  \begin{align*}
    \|(D_x{f}^N)^{\wedge n}\|_{op} &= \frac{\|\partial_1\wedge\cdots\wedge\partial_n\|_{k,N}}{\|\partial_1\wedge\cdots\wedge\partial_n\|_{k,0}} \\
    &= \frac{\exp\left(\frac12\sum_{i=1}^n\sigma_i^{(k)}(x,Nh)\right)}{1},
  \end{align*}
  so taking the logarithm we have our claim.
\end{proof}

\begin{proof}[Proof of Proposition \ref{lem:08}]
The quantity $\log\left\|(D_x{f}^N)^{\wedge n}\right\|_{op}$ gains the interest because of its relation with Lyapunov exponents \cite[\S{2.1}]{ruelle1979}. For $\mu$-a.e. $x$, we have
\begin{equation}
  \label{eqn:5.9} 
  \lim_{N\to\infty}\frac{1}{N}\log\|(D_x{f}^N)^{\wedge n}\|_{op}=\chi_1+\cdots+\chi_{n}=\frac{nh}{2}.
\end{equation}
Now setting
\begin{equation}
  \label{eqn:5.10}
  I_N:=\int_X\log\|(D_x{f}^N)^{\wedge n}\|_{op}\,d\mathrm{vol}(x),
\end{equation}
we have, by Proposition \ref{lem:12},
\begin{equation}
  \label{eqn:5.11}
  I_N =\frac12\int_X\sigma_1^{(k)}(x,Nh)+\cdots+\sigma_n^{(k)}(x,Nh)\,d\mathrm{vol}(x).
\end{equation}
The integrals $\frac{1}{N}I_N$'s have nonnegative integrands. 
By Fatou's Lemma $\int\liminf f_n\leq\liminf\int f_n$, we have:
\begin{align}
  \liminf_{N\to\infty}\frac{1}{N}I_N &\geq\int_X\liminf_{N\to\infty}\frac{1}{N}\log\|(D_x{f}^N)^{\wedge n}\|_{op}\,d\mathrm{vol}(x) \nonumber \\
  &=\frac{nh}{2}, \label{eqn:5.12} 
\end{align}
where the last line is due to \eqref{eqn:5.9}.

Now the inequality
$\|(D_x{f}^{N+M})^{\wedge n}\|_{op}\leq\|(D_x{f}^N)^{\wedge n}\|_{op}\cdot\|(D_{{f}^N(x)}{f}^M)^{\wedge n}\|_{op}$
induces the subadditivity $I_{N+M}\leq I_N+I_M$. By Fekete's Lemma \cite{Fekete1923},
\[
  \inf_{N\geq 1}\frac{1}{N}I_N=\lim_{N\to\infty}\frac{1}{N}I_N,
\]
thus for any $N$, we have
\[
  \frac{1}{N}I_N\geq\lim_{N\to\infty}\frac{1}{N}I_N\geq\frac{nh}{2}.
\]
This finishes the proof, thanks to \eqref{eqn:5.11}.
\end{proof}

\subsection{Jensen's Inequality and Constant Log-singular Values}
\label{sec:jensen-inequality-constant-log-singular-values}

Now we demonstrate how the limit \eqref{eqn:4.6} and Proposition \ref{lem:08} are combined. 
The trick is to use Jensen's inequality, combined with the (strong) convexity of the
$\cosh$ function.

The upshot of this combination is a result on the log-singular values of the `limit metrics' $({f}^N)^\ast\omega_0$, relative to $\omega_0$ (stated in Corollary \ref{lem:14}). By this, we get some simple local representations of these metrics.

Let $B$ be a probability space, whose underlying space is 
$(X\setminus E)\times\{1,\cdots,n\}$, and whose probability measure is 
$\mathrm{vol}\times(\frac{1}{n}\#)$, where $\#$ is the counting measure. 
For $(x,j)\in B$, define a random variable $\Sigma^{(k)}$ as 
$\Sigma^{(k)}(x,j)=\sigma_j^{(k)}(x,Nh)$.

Then the limit \eqref{eqn:4.6} can be rewritten as
\begin{equation}
  \label{eqn:6.1}
  \mathbb{E}[\cosh(\Sigma^{(k)})]-\cosh(Nh)\xrightarrow{k\to\infty}0,
\end{equation}
and the inequality \eqref{eqn:5.2} in Proposition \ref{lem:08} can be rewritten as
\begin{equation}
  \label{eqn:6.2}
  \mathbb{E}[\Sigma^{(k)}]\geq Nh.
\end{equation}

To motivate what follows, we note that Jensen's inequality applied to the 
convex function $\cosh$ gives: 
$\mathbb{E}[\cosh(\Sigma^{(k)})]\geq\cosh(\mathbb{E}[\Sigma^{(k)}])\geq\cosh(Nh)$.
Then \eqref{eqn:6.1} implies that the inequality asymptotically collapses 
as $k\to\infty$. 
There we want to conclude that $\Sigma^{(k)}$ and $Nh$ are approximately the same. 
This idea is precised in
\begin{proposition}
  \label{lem:13}
  As $k\to\infty$,
  the variance of $\Sigma^{(k)}$ is converging to 0,
  and the expected value of $\Sigma^{(k)}$ is converging to $Nh$. 
  That is,
  \[\int_X\frac{1}{n}\sum_{j=1}^n\left(\sigma^{(k)}_j(x,Nh)-Nh\right)^2\,d\mathrm{vol}(x)\xrightarrow{k\to\infty}0.\]
\end{proposition}
\begin{proof}
  We start with an elementary inequality, 
  which holds for any $x,a\in\R$:
  \begin{equation}
    \label{eqn:6.3}
    \cosh(x) \geq \cosh(a) + \sinh(a)\cdot(x-a) + \frac{1}{2}(x-a)^2.
  \end{equation}
  Apply $x=\Sigma^{(k)}$ and $a=Nh$ into \eqref{eqn:6.3}. 
  Taking the average, we get
  \begin{align*}
    \mathbb{E}[\cosh(\Sigma^{(k)})] &\geq \cosh(Nh) + \sinh(Nh)\cdot\underbrace{\left(\mathbb{E}[\Sigma^{(k)}]-Nh\right)}_{\geq 0\text{ by \eqref{eqn:6.2}}} + \frac{1}{2}\mathbb{E}[(\Sigma^{(k)}-Nh)^2] \\
    &\geq \cosh(Nh) + \frac{1}{2}\mathbb{E}[(\Sigma^{(k)}-Nh)^2].
  \end{align*}
  Thus
  \[0\leq\mathbb{E}[(\Sigma^{(k)}-Nh)^2]\leq 2\cdot\left( \mathbb{E}[\cosh(\Sigma^{(k)})]-\cosh(Nh)\right)\xrightarrow{k\to\infty}0,\]
  where we have used the limit \eqref{eqn:6.1}. This implies 
  $\mathbb{E}[(\Sigma^{(k)}-Nh)^2]\to 0$. 
  Our proposition restates this limit.
\end{proof}

Passing to a subsequence of $\left(\omega_k\right)$ if necessary, we further have that 
$\sigma^{(k)}_j(x,Nh)\to Nh$, for $\mathrm{vol}$-a.e. $x$.

What Proposition \ref{lem:13} concludes for the metric $\omega_0$, the limit metric of $\omega_k$'s (see Lemma \ref{lem:01}) defined on $X\setminus E$, is the following

\begin{corollary}
  \label{lem:14}
  For each $x\in X\setminus E$, the log-singular values of 
  $({f}^N)^\ast\omega_0$ relative to $\omega_0$ are $Nh$ and $-Nh$, counted
  $n$ times respectively. 

  That is, for each $x\in X\setminus E$, one can find a holomorphic coordinate 
  $(z_1,\cdots,z_{2n})$ in which the following expressions hold on the tangent space at $x$.
  \begin{align}
    \omega_0 &= \frac{\sqrt{-1}}{2}\sum_{i=1}^{2n}dz_i\wedge d\overline{z}_i, \label{eqn:special-local-coordinates-1} \\
    ({f}^N)^\ast\omega_0 &= \frac{\sqrt{-1}}{2}\sum_{\mu=1}^n e^{Nh}\, dz_\mu\wedge d\overline{z}_\mu + e^{-Nh}\, dz_{n+\mu}\wedge d\overline{z}_{n+\mu}. \label{eqn:special-local-coordinates-2}
  \end{align}
\end{corollary}

We recall that the log-singular values discussed above are eigenvalues of some matrix $\left(h^{(k)}_{ij}\right)$ defined at \eqref{eqn:3.3}. This matrix is defined with a choice of local coordinates, hence \emph{a priori} one cannot expect any regularity for log-singular values. However, one can carefully make a choice of coordinates based on the smoothness of the metric $\omega_0$ and $C^\infty_{\mathrm{loc}}$ convergence of metrics $\omega_k\to\omega_0$. With this, one can use the continuity argument to ensure that the log-singular values of $\omega_0$ are $\pm Nh$.

\begin{proof}
  Fix $N>0$. 
  Let $E'$ denote the union of $E$ and the set of $x\in X\setminus E$ in which the limit $\sigma^{(k)}_j(x,Nh)\to Nh$ as $k\to\infty$ \emph{fails}. We recall that $E'$ necessarily has the zero volume; so $X\setminus E'$ is dense in $X$.

  Because $\omega_0$ is smooth, for all $x_0\in X\setminus E$ we have a neighborhood $U$ of $x_0$ where one can find smooth $(1,0)$-vector fields $(V_i(x))_{i=1}^{2n}$ such that $\omega_0(V_i(x),V_j(x))=\delta_{ij}$. On each point $x\in U$, one can find a holomorphic coordinate $(z_1^{(x)},\ldots,z_{2n}^{(x)})$ such that $\left.\dfrac{\partial}{\partial z_i^{(x)}}\right|_{x}=V_i(x)$ and $\omega_{0,x}=\frac12\sqrt{-1}\sum_{i=1}^{2n}dz_i^{(x)}\wedge d\overline{z}^{(x)}_i$, at $x$ (and possibly no more). Then the matrix $\left(h_{ij}(x)\right)$ defined by
  \begin{align*}
    h_{ij}(x)&=(f^N)^\ast\omega_0\left(\frac{\partial}{\partial z_i^{(x)}},\frac{\partial}{\partial z_j^{(x)}}\right) \\
    &= \left((f^N)^\ast\omega_0\right)_x(V_i(x),V_j(x))
  \end{align*}
  vary smoothly in $x$. As eigenvalues behave continuously by the perturbation of a matrix \cite[Theorem II.5.1]{kato2013}, if we can show that for all $x\in U\setminus E'$ the matrix $\left(h_{ij}(x)\right)$ has eigenvalues $e^{Nh}$ and $e^{-Nh}$, counted $n$ times for each, then the same property holds for all $x\in U$.

  Suppose $x\in U\setminus E'$. 
  As $\left(\omega_k\right)\to\omega_0$ in 
  $C^\infty_{\mathrm{loc}}(X\setminus E)$ topology, focusing on the 
  compact set $\{x,{f}^N(x)\}$, we see that the matrices 
  $\left(h^{(k)}_{ij}(x)\right)$ (see \eqref{eqn:3.3}) at $x$ converge to the analogous matrix $\left(h_{ij}(x)\right)$ for $\omega_0$ at $x$. (An appropriate choice of ``frame'' vector fields $(V^{(k)}_i(x))\to(V_i(x))$ should guarantee this.)
  Again, as eigenvalues behave continuously by the perturbation, 
  the numbers $\exp\left(\sigma^{(k)}_j(x,Nh)\right)$'s approximate 
  eigenvalues of $\left(h_{ij}(x)\right)$. Because $\sigma^{(k)}_j(x,Nh)\to\pm Nh$ as $k\to\infty$, we see that $\left(h_{ij}(x)\right)$ has eigenvalues 
  $e^{Nh}$ and $e^{-Nh}$, counted $n$ times for each.
\end{proof}

\section{Stable and Unstable Manifolds}
\label{sec:stable-unstable-distributions}

The local expressions for $({f}^N)^\ast\omega_0$ and $\omega_0$ (Corollary \ref{lem:14}) verify that ${f}^N$ is expanding and contracting along certain directions in a uniform rate. Initially, these directions are possibly dependent on the time $N$, but one can show that they are actually time independent (Lemma \ref{lem:15}). By this fact, we establish the uniform hyperbolicity for ${f}$ on $X\setminus E$ (Proposition \ref{lem:16}).

This is perhaps one of the rarest moment where one can describe Oseledets splitting (cf. \cite[Theorem 2.2.6]{Fil17}\cite[Theorem 1.6]{ruelle1979}) without limits and thus can verify that it is smooth. But we have a better fact: the distributions define \emph{holomorphic} foliations (Proposition \ref{lem:17}). One can establish this using upper and lower estimates on the growth rate of $f^N$ along stable or unstable distributions, which is typically more than what we know even in uniformly hyperbolic settings.

\subsection{Stable and Unstable Distributions}

Corollary \ref{lem:14} yields that, for every point $x\in X$ and $N\in\Z$, there exists $n$-dimensional subspaces $E^{+N}_{x},E^{-N}_{x}\subset T^{1,0}_xX$ such that every $v\in E^{\pm N}_{x}$ has $\left((f^N)^\ast\omega_0\right)_x(v,v)=e^{\pm Nh}\omega_{0,x}(v,v)$.

We first show that these subspaces $E^{\pm N}_{x}$ do not depend on $N$. We expect this because log-singular values of $(f^N)^\ast\omega_0$ uniformly cumulate by $\pm h$ as we proceed $N$; to have the `optimal cumulation,' we find that the directions that expand by $e^{Nh}$ in $(f^N)^\ast\omega_0$ should also expand by $e^{(N+1)h}$ in $(f^{N+1})^\ast\omega_0$, and similarly for contracting directions.

\begin{lemma}
  \label{lem:15}
  Denote $\|\cdot\|_j$ for the norm associated to the metric $(f^j)^\ast\omega_0$. For any $x\in X\setminus E$ and $N\in\Z_{>0}$, define the following subsets $E^{+N}_x,E^{-N}_x$ and subspaces $F^{+N}_x,F^{-N}_x$ in the holomorphic tangent space $T^{1,0}_xX$:
  \begin{align}
    E^{\pm N}_x &= \{v\in T^{1,0}_xX : \|v\|_N = e^{\pm Nh/2}\|v\|_0\}, \label{eqn:7.2} \\
    F^{\pm N}_x &= \{v\in T^{1,0}_xX : (f^N)^\ast\omega_0(v,w)=e^{\pm Nh}\omega_0(v,w)\ \forall w\in T_xX\}. \label{eqn:7.1}
  \end{align}
  Then we have $E^{+N}_x=F^{+N}_x=E^{+1}_x=F^{+1}_x$ and $E^{-N}_x=F^{-N}_x=E^{-1}_x=F^{-1}_x$ for all $N>0$ and $x\in X\setminus E$. Furthermore, the distributions $E^{\pm 1}$ defined in these fashions are $f$-invariant, i.e., $D_xf(E^{\pm 1}_x)=E^{\pm 1}_{f(x)}$, and have complex dimensions $n$.
\end{lemma}
\begin{proof}
  Fix $N>0$. Fix a holomorphic coordinate $(z_1,\cdots,z_{2n})$ at $x$ that 
  appears in the conclusion of the Corollary \ref{lem:14}. The corresponding holomorphic vectors will have shorthand notations
  \[\partial_i := \frac{\partial}{\partial z_i}.\]
  Denote $G^{+N}_x$ for the (complex) span of the vectors $\partial_1,\ldots,\partial_n$, and $G^{-N}_x$ for the span of the vectors $\partial_{n+1},\ldots,\partial_{2n}$.

  We first claim that $E^{+N}_x=F^{+N}_x=G^{+N}_x$ and $E^{-N}_x=F^{-N}_x=G^{-N}_x$. We show $F^{+N}_x\subset E^{+N}_x\subset G^{+N}_x\subset F^{+N}_x$ to establish the former; the latter can be dealt similarly. Verifications of the inclusions follow.
  \begin{itemize}
  \item Let $v\in F^{+N}_x$. In the identity $(f^N)^\ast\omega_0(v,w)=e^{Nh}\omega_0(v,w)$, we plug in $w=v$. By that we obtain $\|v\|_N^2=e^{Nh}\|v\|_0^2$, so $v\in E^{+N}_x$.
  
  \item Let $v\in E^{+N}_x$. Decompose $v=v^++v^-$ where $v^\pm\in G^{\pm N}_x$. Since both $\omega_0$ and $(f^N)^\ast\omega_0$ view that $G^{\pm N}_x$ are orthogonal, we have
  \begin{align}
    \|v\|_N^2 &= \|v^+\|_N^2+\|v^-\|_N^2. \nonumber
    \intertext{By \eqref{eqn:special-local-coordinates-2}, we can directly evaluate $\|v^\pm\|_N^2$ relative to $\|v^\pm\|_0^2$ and obtain}
    &= e^{Nh}\|v^+\|_0^2 + e^{-Nh}\|v^-\|_0^2 \nonumber \\
    &= e^{Nh}\|v\|_0^2 + (e^{-Nh}-e^{Nh})\|v^-\|_0^2. \label{eqn:7.5-1}
  \end{align}
  Thus $\|v\|_N=e^{Nh/2}\|v\|_0$ implies $\|v^-\|_0=0$. Hence $v=v^+\in G^{+N}_x$.

  \item Let $v\in G^{+N}_x$. Then for any $w^+\in G^{+N}_x$, we have
  \[(f^N)^\ast\omega_0(v,w^+) = e^{Nh}\omega_0(v,w^+)\]
  (by \eqref{eqn:special-local-coordinates-2}). On the other hand, for any $w^-\in G^{-N}_x$, as this is orthogonal to $G^{+N}_x$, we have
  \[(f^N)^\ast\omega_0(v,w^-) = 0 = e^{Nh}\omega_0(v,w^-).\]
  These verify $v\in F^{+N}_x$.
  \end{itemize}

  We note that, as $G^{\pm N}_x$ have dimensions $n$, so are $E^{\pm N}_x$. Furthermore, by the equation \eqref{eqn:7.5-1}, we have an inequality
  \begin{equation}\label{eqn:7.5}\|v\|_N\leq e^{Nh/2}\|v\|_0,\end{equation}
  with equality if and only if $v\in E^{+N}_x$.

  It remains to show that $E^{+N}_x=E^{+1}_x$, for $N>1$. Let $v\in E^{+N}_x$. Then,
  \begin{align}
    \frac{Nh}{2}=\log\frac{\|v\|_{N}}{\|v\|_0}
    &=\log\frac{\|v\|_{N}}{\|v\|_{1}}+\log\frac{\|v\|_{1}}{\|v\|_0} \nonumber\\
    &=\log\frac{\|{f}_\ast v\|_{N-1}}{\|{f}_\ast v\|_0}+\log\frac{\|v\|_1}{\|v\|_0} \nonumber\\
    &\leq\frac{(N-1)h}{2}+\frac{h}{2}=\frac{Nh}{2}. \label{eqn:7.11} 
  \end{align}
  In \eqref{eqn:7.11}, we have used the inequality \eqref{eqn:7.5}. 
  Comparing two sides, we find that we have the equality for \eqref{eqn:7.11}. But then the equality condition of \eqref{eqn:7.5} yields that (a) $v\in E^{+1}_x$ and (b) $f_\ast v=D_xf(v)\in E^{+(N-1)}_{f(x)}$. Therefore (a$'$) $E^{+N}_x\subset E^{+1}_x$ and (b$'$) $D_xf(E^{+N}_x)\subset E^{+(N-1)}_{f(x)}$ hold, which turn out to be equalities thanks to dimension comparisons. This establishes both $E^{+N}_x=E^{+1}_x$ for all $N>1$ and the $f$-invariance.
  
  That $E^{-N}_x=E^{-1}_x$ and their ${f}$-invariance are shown in a similar way, but we need to change \eqref{eqn:7.5} to $\|v\|_0\leq e^{Nh/2}\|v\|_N$, with equality iff $v\in E^{-N}_x$. A further change is required for \eqref{eqn:7.11}, say starting with $Nh/2=\log(\|v\|_0/\|v\|_N)$ where $v\in E^{-N}_x$.
\end{proof}

Denote $E^+,E^-$ for distributions $E^{+1},E^{-1}$ in the above Lemma, respectively. They are unstable and stable distributions, respectively.

\begin{proposition}
  \label{lem:16}
  We have the following operator norms, evaluated at $x\in X\setminus E$ with respect to the metric $\omega_0$:
  \[\|D{f}^N|E^\pm\|_{op}=e^{\pm Nh/2},\quad\text{ and }\quad\|D{f}^{-N}|E^\pm\|_{op}=e^{\mp Nh/2},\]
  where $N\in\Z$ is any integer, including nonpositive ones. In particular, ${f}$ is uniformly hyperbolic on $X\setminus E$, and
  $E^+$, $E^-$ respectively denote unstable and stable distributions
  on $X\setminus E$.
\end{proposition}
\begin{proof}
  We have shown that $\|D{f}^N|E^\pm\|_{op}=e^{\pm Nh/2}$ in Lemma \ref{lem:15}, for $N>0$. 
  It remains to show the same identity for $N<0$. 
  As usual, for $j\in\Z$, we denote $\|\cdot\|_j$ by the norm associated to the metric $({f}^j)^\ast\omega_{0}$.

  For $N<0$, let $M=-N>0$. To estimate $\|D{f}^{-M}|E^\pm\|_{op}$, we pick up $v\in E^\pm_x$ and estimate 
  $\|v\|_{-M}/\|v\|_0$. Here, by definition of $({f}^{-M})^\ast\omega_{0}$, 
  we have
  \[\left(\frac{\|v\|_N}{\|v\|_0}=\right)\frac{\|v\|_{-M}}{\|v\|_0}=\frac{\|{f}^{-M}_\ast v\|_0}{\|{f}^{-M}_\ast v\|_{M}}=\frac{1}{e^{\pm Mh/2}}=e^{\pm Nh/2},\] 
  since ${f}^{-M}_\ast v\in E^{\pm}_{{f}^{-M}(x)}$ as well.
  This shows the claim.
\end{proof}

The expression \eqref{eqn:7.1} shows that $E^\pm$ are characterized
by $C^\infty$ conditions (essentially because $\omega_0$ is smooth).
Therefore they are $C^\infty$ distributions.
One can appropriately multiply matrices for ${f}^\ast\omega_0$ and
$\omega_0^{-1}$ to find explicit descriptions for
$C^\infty$ vector field generators of them. 
We then have the following

\begin{proposition}
  \label{lem:doubly-foliating-chart}
  For each $p\in X\setminus E$ there is a real $C^\infty$ chart $(V;x_1^+,\ldots,x^+_{2n},x_1^-,\ldots,x_{2n}^-)$ near $p$ such that
  \begin{itemize}
    \item on the coordinate neighborhood $V$, $E^\pm=\bigcap_{i=1}^{2n}\ker(dx^\pm_i)$, and
    \item the slice manifolds $\{x_i^-=c_i,\forall i=1,\ldots,2n\}\subset V$ are open subsets of local stable manifolds in $V$ and likewise for local unstable manifolds.
  \end{itemize}
\end{proposition}
By the second item, we say $x^+_1,\ldots,x^+_{2n}$ the \emph{stable coordinates} (as they form coordinates on stable manifolds) and $x^-_1,\ldots,x^-_{2n}$ the \emph{unstable coordinates}.
\begin{proof}
  By Frobenius theorem \cite[Theorem 19.12]{LeeSmooth}, one can obtain (real $C^\infty$) foliating charts $(\hat{V},\mathbf{x}^\pm)$ for $E^\pm$ at $p$, where $\mathbf{x}^\pm=(x^\pm_1,\ldots,x^\pm_{4n})$ and $E^\pm=\bigcap_{i=1}^{2n}\ker(dx^\pm_i)$. Let $\mathbf{x}=(x_1^+,\ldots,x_{2n}^+,x_1^-,\ldots,x_{2n}^-)$ be a map $\mathbf{x}\colon\hat{V}\to\R^{4n}$. Then its derivative has the kernel $\bigcap_{i=1}^{2n}(\ker(dx^+_i)\cap\ker(dx^-_i))=E^+\cap E^-=0$. Hence $\mathbf{x}$ locally defines smooth coordinates, say on $V\Subset\hat{V}$ a precompact neighborhood of $p$. 


  For each point $q\in V$, denote $\zeta^-(q)=\{y\in V : (\forall 1\leq i\leq 2n)(x^-_i(y)=x^-_i(q))\}$. We claim that $\zeta^-(q)$ is an open neighborhood of $q$ in $W^-_{\mathrm{loc}}(q)$; here, $W^-_{\mathrm{loc}}(q)$ is the local stable manifold of $q$ in $V$ in the sense of \cite[Theorem 6.1(a)]{ruelle1979}. In particular, we show that there is a constant $C(V)>0$ depending only on $V$ such that for every $y\in \zeta^-(q)$ and $N>0$, we have
  \begin{equation}
    \label{eqn:slice-local-stable-manifold}
    \mathrm{dist}_{\omega_0}(f^Ny,f^Nq)\leq C(V)e^{-Nh/2}.
  \end{equation}
  To see this, consider the curve $\gamma(t)=\mathbf{x}^{-1}((1-t)\mathbf{x}(q)+t\mathbf{x}(y))$ on $V$. Then we have $\gamma'(t)\in E^-$ for all $t\in[0,1]$ and the length of $\gamma$ is bounded above by a constant $C(V)$ that depends on the diameter of $\mathbf{x}(V)\Subset\R^{4n}$ and a bound for $\omega_0$ on $V$. Sending the curve $\gamma$ by $f^N$, we see that the length of the curve $f^N\circ\gamma$ is at most $e^{-Nh/2}$ times the length of $\gamma$. Thus we have the inequality \eqref{eqn:slice-local-stable-manifold}.
\end{proof}

\subsection{Holomorphicity of the Stable Distribution}
\label{sec:holomorphicity-stable-distribution}

So far we have studied about the stable and unstable distributions, and concluded that 
they are defined in a $C^\infty$ manner. 
What can be claimed further is that, the stable and unstable foliations 
are actually holomorphic. 

As shown in \cite[Lemma 2.1]{ghys1995}, each leaf of the foliations $W^\pm$ 
generated by $E^\pm$ (respectively), are holomorphic manifolds. 
In particular, there are unstable vector fields which are holomorphic along 
unstable manifolds, and it thus remains to show its holomorphicity 
along the transverse direction. 
(Similar claim may be made for stable vector fields.)

The trick is to use the \emph{Poincar\'e map}, as described in 
\cite[\S{}III.3]{mane1987}. 
Let $U$ and $U'$ be local unstable manifolds, close and small enough so that one can define a Poincar\'e map $\phi\colon U\to U'$. 
Denote $I$ for the complex structure of $X$. We will use $I_y$ to denote the real linear map $T_{\R,y}X\to T_{\R,y}X$ that it 
induces at a point $y\in X$. 
Define $[D\phi,I]$, a family of maps $[D\phi,I]_x\colon T_xU\to T_{\phi(x)}U'$, as
\[[D\phi,I]_x := D_x\phi\circ I_x - I_{\phi(x)}\circ D_x\phi.\]
If $[D\phi,I]=0$ can be shown, then it shows that Poincar\'e maps 
are holomorphic, showing the desired claim.

This goal setup is encoded into the following
\begin{proposition}
  \label{lem:17}
  Let $U,U'$ be local unstable manifolds, not intersecting with one another, 
  and close enough to induce a Poincar\'e map $\phi\colon U\to U'$. 
  Then $[D\phi,I]_x=0$ for all $x\in U$.
\end{proposition}

\subsubsection{On a Foliated Chart}

Fix a chart $(V,\mathbf{x})$ as described in Proposition \ref{lem:doubly-foliating-chart}. 
Now suppose $x$ and $\phi(x)$ are in $V$. 
Then a neighborhood of $x\in U$ and $\phi(x)\in U'$ are laid along the
coordinate directions for $E^-$, 
and hence $\phi$ near $x$ is described as a translation of stable coordinates.
This concludes that $D\phi$ is represented as the identity matrix,
with respect to unstable coordinates on $U,U'$.

The complex structure $I$ is understood as a family of linear maps on 
real tangent spaces $T_{\R,y}X\to T_{\R,y}X$. So if we fix coordinates, 
say the foliated coordinates on $V$, we have a matrix representation 
$I_y\in\mathrm{GL}(4n,\R)$ for each $y\in V$.
Collecting these remarks, we conclude as follows.

\begin{lemma}
  \label{lem:18}
  Let $U,U'$ be local unstable manifolds as in Proposition \ref{lem:17}, 
  and suppose $x\in U$ is such that $x,\phi(x)\in V$.
  For the matrix representation $\{I_y\}_{y\in V}$ of the 
  complex structure $I$ on $V$, we have
  \begin{equation}
    \label{eqn:8.1}
    [D\phi,I]_x = I_x - I_{\phi(x)}.
  \end{equation}
\end{lemma}

Furthermore, shrinking $V$ if necessary, we assume that the family 
$\{I_y\}_{y\in V}$ satisfies the Lipschitz condition, i.e.,
\begin{equation}
  \label{eqn:8.2} 
  \|I_p-I_q\|\leq C\cdot\mathrm{dist}(p,q).
\end{equation}

\subsubsection{From Ergodicity}

Apparently, $(V,\mathbf{x})$ is set on an arbitrary place of $X$, 
thus we hardly have $x,\phi(x)\in V$ in general. 
We have shown that, in Proposition \ref{lem:ergodicity}, the Green measure (the unique measure of maximal entropy) is ergodic 
with respect to ${f}\colon X\to X$, and we have assumed that the Green measure 
is the volume measure. We then claim the following.
\begin{lemma}
  \label{lem:19}
  Let $U,U'$ be local unstable manifolds as in Proposition \ref{lem:17}. Then for $\mathrm{vol}$-a.e. $x$, there are infinitely many $N$ such that ${f}^N(x),{f}^N\phi(x)\in V$.
\end{lemma}
\begin{proof}
  Let $V'\Subset V$ be a nonempty precompact open subset.
  Let $\epsilon=\inf_{y\in V'}\mathrm{dist}(y,\partial V)$ be the minimum (K\"ahler) distance from $V'$ to the complement of $V$, which is a positive number.

  By the uniform contraction of $E^-$ along ${f}$, we have
  \begin{equation}
    \label{eqn:8.3} 
    \mathrm{dist}_{E^-}({f}^N(x),{f}^N\phi(x))\leq C(U,U',V)\cdot e^{-Nh/2}\mathrm{dist}_{E^-}(x,\phi(x)),
  \end{equation}
  where $\mathrm{dist}_{E^-}$ is the distance measured along the stable leaf, and $C(U,U',V)>0$ is a constant depending only on $U,U'$ and $V$.
  As $\mathrm{dist}(p,q)\leq\mathrm{dist}_{E^-}(p,q)$, 
  this tells that $\mathrm{dist}({f}^N(x),{f}^N\phi(x))<\epsilon$ 
  for sufficiently large $N\geq N_0$. 
  In particular, if ${f}^N(x)\in V'$, $N\geq N_0$, then 
  ${f}^N(x),{f}^N\phi(x)\in V$.
  
  Now because $V'$ is a nonempty open set, it has a nonzero volume. 
  Thanks to Birkhoff ergodicity, we conclude that there are infinitely many
  $N$ such that ${f}^N(x)\in V'$ for $\mathrm{vol}$-a.e. $x$. The lemma then follows.
\end{proof}

\subsubsection{Future Estimates}

To prove Proposition \ref{lem:17}, we use the trick of `sending to the future,' 
as commonly seen in \cite[Theorem 2.2]{ghys1995} and 
\cite[Theorem III.3.1]{mane1987}.
The trick starts from the following decomposition:
\begin{equation}
  \label{eqn:8.4} 
  [D\phi,I]_x=D_{{f}^N\phi(x)}{f}^{-N}\circ [D({f}^N\phi {f}^{-N}),I]_{{f}^N(x)}\circ D_x{f}^N.
\end{equation}
We then estimate each factor:
$(D{f}^{-N}|T{f}^N(U'))=D{f}^{-N}|E^+$, $(D{f}^N|TU)=D{f}^N|E^+$, 
and $[D({f}^N\phi {f}^{-N}),I]$. (Below, $C_1,C_2>0$ are constants that only depend on the Poincar\'e map $\phi$.)

\begin{itemize}
  \item For $D{f}^{-N}|E^+$, we use Proposition \ref{lem:16} to have that $D{f}^{-1}$ is (under $\omega_0$)
  uniformly contracting on $E^+$ with the rate $e^{-h/2}$. Applying such, we get
  \begin{equation}
    \label{eqn:8.5} 
    \|D{f}^{-N}|E^+\|\leq C_1\cdot e^{-Nh/2}.
  \end{equation}

  \item For $D{f}^N|E^+$, we use Proposition \ref{lem:16} to have that $D{f}$ is (under $\omega_0$)
  uniformly expanding on $E^+$ with the rate $e^{h/2}$. Applying such, we get
  \begin{equation}
    \label{eqn:8.6} 
    \|D{f}^N|E^+\|\leq C_2\cdot e^{Nh/2}.
  \end{equation}

  \item Finally, for $[D({f}^N\phi {f}^{-N}),I]$, we pick $N$ such that
  ${f}^N(x),{f}^N\phi(x)\in V$ by Lemma \ref{lem:19}. (Note that this may be done only for $\mathrm{vol}$-a.e. $x$.)
  
  The map ${f}^N\phi {f}^{-N}$ is a Poincar\'e map ${f}^N(U)\to {f}^N(U')$ along the stable direction.
  Thus Lemma \ref{lem:18} applies to give,
  \begin{align}
    \left\|[D({f}^N\phi{f}^{-N}),I]_{{f}^N(x)}\right\| &= \left\|I_{{f}^N(x)} - I_{{f}^N\phi(x)}\right\| \nonumber \\
    &\leq C\cdot\mathrm{dist}({f}^N(x),{f}^N\phi(x)). \nonumber
    \intertext{Via the estimate \eqref{eqn:8.3}, we further estimate,}
    &\leq C\cdot\mathrm{dist}_{E^-}({f}^N(x),{f}^N\phi(x)) \nonumber \\
    &\leq C\cdot e^{-Nh/2}\mathrm{dist}_{E^-}(x,\phi(x)). \label{eqn:8.7} 
  \end{align}
\end{itemize}

Combining all three estimates \eqref{eqn:8.5}, \eqref{eqn:8.6}, and 
\eqref{eqn:8.7}, we obtain, in \eqref{eqn:8.4},
\[\|[D\phi,I]_x\|\leq C_1C_2C\cdot e^{-Nh/2}\mathrm{dist}_{E^-}(x,\phi(x)),\]
whenever $N$ satisfies ${f}^N(x),{f}^N\phi(x)\in V$.
For $\mathrm{vol}$-a.e. $x$, there are infinitely many such $N$'s. Sending $N\to\infty$, 
we have $[D\phi,I]_x=0$, $\mathrm{vol}$-a.e. $x$. Appealing to the continuity of $x\mapsto[D\phi,I]_x$, we prove the Proposition \ref{lem:17}.

\subsection{Flatness}
\label{sec:flatness}

Thanks to the holomorphicity of stable and unstable foliations, we have the following flatness result. This serves as a key ingredient to infer that the initial manifold $X$ is induced from a torus.

We note that we can apply the proof below for K3 surfaces and get a shorter proof of some known facts like \cite[Proposition 3.2.1]{FT18}.

\begin{proposition}
  \label{lem:20}
  The metric $\omega_0$ on $X\setminus E$ is flat.
\end{proposition}
\begin{proof}
  Fix a point $x\in X\setminus E$. 
  We know that $E^-$ and $E^+$ are holomorphic distributions.
  By a construction similar to the proof of Proposition \ref{lem:doubly-foliating-chart}, 
  one can find a holomorphic chart $(w^+_1,\ldots,w^+_n,w^-_1,\ldots,w^-_n)$ near $x$ such that  
  we have $E^\pm=\bigcap_{i=1}^n\ker(dw^\pm_i)$. 

Moreover, it is easy to check that $E^-$ and $E^+$
are orthogonal under $\omega_0$, as follows. For $v\in E^-$ and $w\in E^+$, we get
\[e^{-h}\omega_0(v,w)={f}^\ast\omega_{0}(v,w)=-{f}^\ast\omega_{0}(w,v)=-e^{h}\omega_0(w,v)=e^{h}\omega_0(v,w),\]
and thus $\omega_0(v,w)=0$. Therefore one can re-write $\omega_0$ as
\[\omega_0 = \frac{\sqrt{-1}}{2}\sum_{i,j=1}^n\left(a_{i\overline{j}}\, dw_i^+\wedge d\overline{w}_j^+ + b_{i\overline{j}}\, dw_i^-\wedge d\overline{w}_j^-\right),\]
with some positive-definite matrix-valued functions
$(a_{i\overline{j}})$ and $(b_{i\overline{j}})$.

That $d\omega_0=0$ then implies that the 
$w_1^-,\overline{w}_1^-,\cdots,w_n^-,\overline{w}_n^-$-derivatives of $a_{i\overline{j}}$ shall vanish and the 
$w_1^+,\overline{w}_1^+,\cdots,w_n^+,\overline{w}_n^+$-derivatives of $b_{i\overline{j}}$ shall vanish.
Consequently, $\omega_0$ is split completely into:
\[\omega_0 = \omega_0^-(w_1^+,\cdots,w_n^+) + \omega_0^+(w_{1}^-,\cdots,w_{n}^-).\]
In short, the metric $\omega_0$ decomposes as $\omega_0=\omega_0^-\times\omega_0^+$ (locally).
%

Now consider the Levi-Civita connection $\nabla$ of the metric $\omega_0$.
This connection satisfies the followings:
\begin{enumerate}
  \item $\nabla\Omega=0$.
As $\omega_k$'s satisfy this, and $\nabla\Omega=0$ is expressed with Christoffel
symbols of the metric, that $\omega_k\to\omega_0$ in $C^\infty_{\mathrm{loc}}$
certifies this for $\omega_0$ as well.
  \item $\nabla E^+\subset E^+$ and $\nabla E^-\subset E^-$.
This follows from the local factorization $\omega_0=\omega_0^-\times \omega_0^+$, 
where each $\omega^\pm_0$ is supported on $E^\pm$, respectively.
  \item For vector fields $Z^+$ on $E^+$ and $Z^-$ on $E^-$, the following holds:
  \begin{align*}
    \nabla_{Z^-}Z^+ &= p^+([Z^-,Z^+]), \\
    \nabla_{Z^+}Z^- &= p^-([Z^+,Z^-]),
  \end{align*}
  where $p^\pm$ denote the parallel projections $E^-\oplus E^+\to E^\pm$.
These follow from the torsion-free property
$[Z^-,Z^+]=\nabla_{Z^-}Z^+-\nabla_{Z^+}Z^-$,
together with that $\nabla E^\pm\subset E^\pm$ verified above.
\end{enumerate}

According to \cite[Lemme 3.4.4]{BFL92}, connections satisfying all three above
are unique. Now by the proof of \cite[Lemme 2.2.3(b)]{BFL92},
we get that $\nabla$ is a flat connection. Hence $\omega_0$ is flat on $X\setminus E$. 
(The cited theorems are all local, thus the non-compact nature of
$X\setminus E$ does not matter here.)
\end{proof}

\section{Proof of Theorem \ref{lem:00}}
\label{sec:upshots-of-flatness}

Now we are almost ready to prove the Theorem \ref{lem:00}. The final steps of the proof require some complex geometric constructions. In particular, based on some results from \cite{GKP} and \cite{BCL14}, and the fact that $X$ is projective, we establish that our hyperk\"ahler manifold $X$ is a modification of a torus quotient. The automorphism $f$ also turns out to induce a rational self-map on a torus, which could be undefined in a codimension 1 subset. However, this suffices to show that $(X,f)$ is a Kummer example \cite[Lemma 1.25]{LoB17}.

\subsection{Complex Geometric Terms}

We briefly list some complex geometric terms used in the proof below.

We recall the notion of normal varieties $Y$ \cite[Exercise I.3.17]{Hartshorne}. We define its \emph{regular locus} $Y_{\mathrm{reg}}$ as the set of nonsingular points in $Y$ \cite[\S{I.5}, p.32]{Hartshorne}. On the regular locus, one can define the \emph{tangent sheaf} $T|Y_{\mathrm{reg}}$ on it \cite[\S{II.8}, p.180]{Hartshorne}, as $Y_{\mathrm{reg}}$ itself is a nonsingular variety. 

We say a rank $r$ vector bundle $E\to Z$ on a complex algebraic variety $Z$ is \emph{flat} if there exists a representation of the topological fundamental group $\rho\colon\pi_1(Z)\to\mathrm{GL}(r,\C)$ such that $E$ is isomorphic to the bundle $\widetilde{Z}\times\C^r/\pi_1(Z)\to Z$, constructed from the diagonal action $\pi_1(Z)\ACTS\widetilde{Z}\times\C^r$ \cite[Definition 2.3]{claudon2020kahler}. This is an analytic paraphrase of the flatness defined in \cite[Definition 1.15]{GKP}.

We say a morphism $\phi\colon X\to Y$ is \emph{\'etale in codimension one} if there exists a closed subset $Z\subset X$ of codimension $\geq 2$ such that $\phi|_{X\setminus Z}\colon X\setminus Z\to Y$ is \'etale \cite[Definition 3.3]{GKP}. 
Likewise, we say a group $\Gamma$ acts on a variety $Y$ \emph{freely in codimension one} if there is a $\Gamma$-invariant closed subset $Z'\subset Y$ of codimension $\geq 2$ such that $\Gamma$ acts freely on $Y\setminus Z'$.

We say a normal variety has \emph{canonical singularities} in the sense of \cite[Definition 2.11]{KollarMori}. We say a normal variety $Y$ has \emph{klt singularities} if the pair $(Y,\varnothing)$ is a divisorial log terminal singularity in the sense of \cite[Definition 2.37]{KollarMori}. Varieties with canonical singularities also have klt singularities.

We say a projective morphism $\phi\colon X\to Y$ is a \emph{symplectic resolution} if $X$ is a smooth variety carrying a closed nondegenerate global 2-form $\sigma\in H^0(X,\Omega^2_X)$ and $Y$ is a normal variety, following \cite[Definition 2]{wierzba}. We say a morphism $\phi\colon X\to Y$ is \emph{crepant} if no discrepancies appear, i.e., we have $\phi^\ast K_Y=K_X$.

\subsection{Flat Regular Locus implies Torus Quotient}

The following fact states that a normal variety must be a torus quotient if its tangent sheaf on the regular locus has a flat metric.
\begin{theorem}
  \label{lem:21}
  Suppose $Y$ is a normal
  complex projective variety that has klt singularities.
  If the tangent bundle $T|Y_{\mathrm{reg}}$ is flat, then
  $Y$ is a quotient of a complex torus
  by a finite group acting freely in codimension one.
\end{theorem}
\begin{proof}
  Apply \cite[Corollary 1.16]{GKP} to the klt pair $(Y,\varnothing)$. As $Y$ is projective, there exists an abelian variety $\mathbb{T}$ and a finite Galois morphism $\pi\colon\mathbb{T}\to Y$ that is \'etale in codimension one. This yields a finite group $\Gamma\subset\mathrm{Aut}(\mathbb{T})$ such that $\pi$ is the quotient morphism $\mathbb{T}\to\mathbb{T}/\Gamma=Y$ (cf. \cite[Definition 3.6]{GKP}).
\end{proof}

In what follows, we construct the normal variety $Y$ to plug in the above
Theorem \ref{lem:21}.
Morally, it is constructed by contracting $E\subset X$ by a contraction
$\phi\colon X\to Y$. Such construction requires
$X$ to be projective.

\subsubsection{Remark}
A result by Claudon, Graf, Guenancia, and Naumann \cite[Theorem D]{claudon2020kahler} gives rise to a generalization of \cite[Corollary 1.16]{GKP}, applicable for compact K\"ahler normal complex spaces with klt singularities (i.e., drops the projectivity assumption). Thus Theorem \ref{lem:21} may be extended to the case of a non-projective $Y$. Nonetheless, in this paper we still need $X$ to be projective to construct $Y$.

\subsection{Construction of the Contraction}

We now aim to prove the following
\begin{proposition}
  \label{lem:contraction-construct}
  There exists a contraction $\phi\colon X\to Y$ such that $Y$ is a normal projective variety and its regular locus $Y_{\mathrm{reg}}$ is the image of $X\setminus E$. Moreover, $Y$ has canonical singularities and has a K\"ahler current $\phi_\ast\omega_0$ on $Y_{\mathrm{reg}}$ which is a flat metric.
\end{proposition}

The proof of this fact extensively uses the fact that $X$ is projective. As a preparation, we first show that the eigenclasses $[\eta_+]$ and $[\eta_-]$ are in fact $(1,1)$-classes of a nef Cartier $\R$-divisor. Appealing to the projectivity of $X$, fix an ample class $[A]$. Then by Proposition \ref{lem:04}(\ref{enum:lem:04:d}), as $n\to\infty$,
\begin{align*}
  \frac{\lambda^{-n}}{2q([A],[\eta_-])}({f}^n)^\ast[A]&\to\binom{2n}{n}^{1/n}[\eta_+]; \\
  \frac{\lambda^{-n}}{2q([A],[\eta_+])}({f}^{-n})^\ast[A]&\to\binom{2n}{n}^{1/n}[\eta_-].
\end{align*}

Therefore $\alpha=[\eta_+]+[\eta_-]$ is also the class of a Cartier $\R$-divisor, which is big and nef. Note that $\alpha$ is not ample; otherwise, its null locus $E=\varnothing$ would be empty, and by Proposition \ref{lem:20} we have $X$ a compact flat manifold. The only such manifolds are tori \cite{BieberbachI}\cite{BieberbachII}, which contradicts to that $X$ is simply connected.

The first step of proving Proposition \ref{lem:contraction-construct} is to construct the contraction $\phi$. This is essentially done by \cite[Theorem A]{BCL14}, but generalized to $\R$-Cartier classes. We present it in the following
\begin{lemma}
  \label{lem:22}
  Let $\alpha$ be the $(1,1)$-class of a big and nef Cartier $\R$-divisor, and $E$ be its null locus \cite[p.1168]{CT15}. Then one can construct a contraction $\phi\colon X\to Y$ so that $X\setminus E$ is the maximal Zariski open subset in which $\phi$ maps it isomorphically onto its image. (That is, $E=\mathrm{Exc}(\phi)$.)
\end{lemma}
\begin{proof}
Denote $\mathrm{Amp}(X)$ and $\mathrm{Big}(X)$ for the cone of $(1,1)$-classes of ample and big Cartier $\R$-divisors, respectively. By Kawamata's Rational Polyhedral Theorem \cite[Theorem 5.7]{Kaw88}, the face $F$ of the cone $\mathrm{Big}(X)\cap\overline{\mathrm{Amp}}(X)$ in which $\alpha$ lies on is represented by a rational linear equation. Consequently, one can write $\alpha=\sum_{\mathrm{finite}}a_i c_1(L_i)$ where each $a_i>0$ and each $L_i$ is a big and nef line bundle in which $c_1(L_i)\in F$.

Because each $L_i$ is big and nef, by basepoint-free theorems \cite[Theorem 3.9.1]{BCHM}\cite[Theorem 1.3]{Kaw88}, it is semiample. Because all $L_i$'s lie on the same face of the big and nef cone $\mathrm{Big}(X)\cap\overline{\mathrm{Amp}}(X)$, the images of the morphism
\[\Phi_{mL_i}\colon X\to\Pp H^0(X,mL_i)\]
are isomorphic to each other, when $m\gg 0$ (cf. \cite[Definition 3-2-3]{KMM}).

For each $L_i$, denote its augmented base loci by $E_i$, denote the image of $\Phi_{mL_i}$ by $Y_i$, and let $\phi_i:=\Phi_{mL_i}|X\setminus E_i$ be a restriction.

We claim that $E_i=E_j$. Fix an isomorphism $\psi\colon Y_i\to Y_j$ in which $\psi\circ\phi_i=\phi_j$. By \cite[Theorem A]{BCL14}, the complement $X\setminus E_i$ of the locus is characterized by the maximal Zariski open subset in which $\Phi_{mL_i}$ isomorphically sends the subset into its image. Postcomposing $\psi$ to $\Phi_{mL_i}$, we see that $X\setminus E_i$ is sent isomorphically into its image via $\Phi_{mL_j}$. Consequently, we have $X\setminus E_i\subset X\setminus E_j$. Arguing symmetrically, we have the claim.

Fix a bundle $L_i$ and denote $\phi:=\phi_i$. An upshot of the above paragraph is that $\phi$ is a contraction in which $X\setminus E_i$ is the maximal Zariski open subset where $\phi$ maps it isomorphically onto its image.

We claim that $E=E_i$. Let $L=\sum L_i$. Denote $E'$ for the augmented base locus of $L$. Then we have (i) $E'=E_i$, by the same token of showing $E_i=E_j$, and (ii) $E'\subset E\subset E_i$, by the followings. (By \cite[Corollary 1.2]{CT15}, it suffices to compare the null loci.)
\begin{description}
  \item[($E\subset E_i$)] Fix any subvariety $V\subset X$ in which $\int_V\alpha^{\dim V}=0$. By $\alpha=\sum a_i c_1(L_i)$ and because the multinomial theorem gives nonnegative terms, we obtain $\int_V c_1(L_i)^{\dim V}=0$. Thus $V\subset E_i$, and $E\subset E_i$ follows.
  \item[($E'\subset E$)] For any subvariety $V\subset X$ in which $\int_V\left(\sum c_1(L_i)\right)^{\dim V}=0$, expand it with the multinomial theorem to have $\int_V \prod c_1(L_i)^{e_i}=0$, whenever $\sum e_i=\dim V$. As $\alpha=\sum a_ic_1(L_i)$, again by multinomial theorem, we have $\int_V\alpha^{\dim V}=0$. This shows $E'\subset E$.
\end{description}
Combining (i) and (ii) we have $E=E_i$, as required.
\end{proof}

Define the \emph{exceptional set} $\mathrm{Exc}(\phi)$ as the minimal Zariski closed subset $E'\subset X$ in which $X\setminus E'$ is mapped isomorphically onto its image by $\phi$. By what is stated in Lemma \ref{lem:22}, that $E=\mathrm{Exc}(\phi)$ paraphrases the statement. This set is, by the inverse function theorem, same as the set of $x\in X$ in which $D_x\phi$ is invertible.

\begin{proof}[Proof of Proposition \ref{lem:contraction-construct}]
Construct the contraction map $\phi\colon X\to Y$ by Lemma \ref{lem:22}, for $\alpha=[\eta_+]+[\eta_-]$. Then $Y_{\mathrm{reg}}=\phi(X\setminus\mathrm{Exc}(\phi))=\phi(X\setminus E)$ follows.

Recall that there is a flat metric $\omega_0$ on $X\setminus E$ (Proposition \ref{lem:20}). Pushforwarding this to $Y_{\mathrm{reg}}$, we have a K\"ahler current $\phi_\ast\omega_0$ which is also a flat metric on $Y_{\mathrm{reg}}$.

To see why $Y$ has canonical singularities, we use the remark in \cite[Remark 1]{wierzba}. As $\phi$ is a symplectic resolution, $\phi$ is crepant, i.e., $\phi^\ast K_Y=K_X$.
\end{proof}

\subsection{Proof of Theorem \ref{lem:00}}
\label{sec:proof-of-thm}

What was claimed about $Y$ in Proposition \ref{lem:contraction-construct} additionally yields that the tangent bundle $T|Y_{\mathrm{reg}}$ is flat, hence the hypotheses of the Theorem \ref{lem:21} above are met.
Indeed, the metric $\phi_\ast\omega_0|Y_{\mathrm{reg}}$ produces a flat connection on $T|Y_{\mathrm{reg}}$. 
Therefore $Y$ is a torus quotient; that is,
there exists a complex torus $\mathbb{T}=\C^{2n}/\Lambda$ and a finite group
of toral isomorphisms $\Gamma$
in which $Y=\mathbb{T}/\Gamma$.

To show that ${f}$ is induced from a hyperbolic linear transform, 
recall that $\phi$ isomorphically sends 
$X\setminus E$ to $Y_{\mathrm{reg}}$.
Conjugating ${f}|_{X\setminus E}$ via $\phi$, we then have a map 
$\widetilde{f}\colon Y_{\mathrm{reg}}\to Y$. 
This $\widetilde{f}$ lifts to a rational map 
$\mathbb{T}\dashrightarrow\mathbb{T}$, defined in codimension 1. 
The only such map is affine-linear \cite[Lemma 1.25]{LoB17}, 
and this descends down to a morphism $\widetilde{f}\colon Y\to Y$. 
This verifies the desired classification of ${f}$, 
and finishes the proof of Theorem \ref{lem:00}.

\bibliography{references}{}
\bibliographystyle{amsalpha}

\end{document}